\documentclass[3p]{elsarticle}
\usepackage{amsthm}
\usepackage{amsmath}%
\usepackage{amsfonts}%
\usepackage{amssymb}%
\usepackage{graphicx}
\usepackage{mathabx}
\usepackage{enumitem}
\usepackage{bbm}
\usepackage{lipsum}

\usepackage{hyperref}

\theoremstyle{remark}
\newtheorem{remark}{Remark}

\theoremstyle{definition}
\newtheorem{definition}{Definition}

\theoremstyle{plain}
\newtheorem{theorem}{Theorem}

\newtheorem{corollary}{Corollary}

\newtheorem{lemma}{Lemma}

\newtheorem{proposition}{Proposition}

\numberwithin{equation}{section}
\numberwithin{theorem}{section}
\numberwithin{proposition}{section}
\numberwithin{lemma}{section}
\numberwithin{corollary}{section}
\numberwithin{remark}{section}
\numberwithin{example}{section}
\numberwithin{definition}{section}

\newcommand{\intt}[1]{\int_0^t#1\;\text{d}s\,}

\newcommand{\norm}[1]{\left\lVert#1\right\rVert}

\journal{Journal of Differential Equations}

\begin{document}
\begin{frontmatter}
\title{Kinetic Relaxation to Entropy Based Coupling Conditions for Isentropic Flow on Networks}
\author{Yannick Holle\fnref{myfootnote}}
\address{Institut f\"ur Mathematik, RWTH Aachen University,\\ Templergraben 55, D-52062 Aachen, Germany}
\ead{holle@eddy.rwth-aachen.de}
\fntext[myfootnote]{This work has been funded by the Deutsche Forschungsgemeinschaft
(DFG, German Research Foundation) Projektnummer 320021702/GRK2326
Energy, Entropy, and Dissipative Dynamics (EDDy). The author would like to thank Michael Herty and Michael Westdickenberg for discussions and bringing the subject to his attention.}
\begin{abstract}
We consider networks for isentropic gas and prove existence of weak solutions for a large class of coupling conditions. First, we construct approximate solutions by a vector-valued BGK model with a kinetic coupling function. Introducing so-called kinetic invariant domains and using the method of compensated compactness justifies the relaxation towards the isentropic gas equations. We will prove that certain entropy flux inequalities for the kinetic coupling function remain true for the traces of the macroscopic solution. These inequalities define the macroscopic coupling condition. Our techniques are also applicable to networks with arbitrary many junctions which may possibly contain circles. We give several examples for coupling functions and prove corresponding entropy flux inequalities. We prove also new existence results for solid wall boundary conditions and pipelines with discontinuous cross-sectional area.
\end{abstract}
\end{frontmatter}

\section{Introduction}
\label{sec:Introduction}
This paper considers networks modeled by one dimensional conservation laws which are coupled at a junction. We are especially interested in (isentropic) gas flows in pipeline networks, but there are many other applications for example in traffic, supply chains, data networks or blood circulation. 
This field became of interest of many researchers in the last two decades and was studied in various directions (analysis, numerics, modeling, optimization,...). See for example the overview by Bressan et al. \cite{BCGHP2014}. In this paper we will rigorously prove existence of solutions to the coupled Cauchy problem. We use a kinetic BGK model to construct approximate solutions and justify the limit with the compensated compactness method. The obtained macroscopic solution satisfies inherited entropy flux inequalities at the junction.\medskip\\
Bouchut \cite{Bo1999} introduced a (vector-valued) BGK model relaxing to the isentropic gas equations. We will use this model to construct a sequence of approximate solutions. Berthelin and Bouchut proved the relaxation of finite mass and energy solutions rigorously for initial value problems \cite{BeBo2000,BeBo2002Kinetic} and initial boundary value problems \cite{BeBo2002Boundary}. The construction of BGK solutions is simple and can be done by a characteristics formula and a fixed point argument. We adopt these techniques to networks with a kinetic coupling condition. \\
To justify the relaxation process, we will use Tartar's method of compensated compactness \cite{Ta1979}. The method can be used for strictly hyperbolic conservation laws with a rich family of entropies. DiPerna \cite{Di1983} adopted this technique to the isentropic gas equations which are not strictly hyperbolic in the vacuum. DiPerna's result holds if the finite mass and energy initial data is bounded in $L^\infty$ and the adiabatic exponent is given by $\gamma=1+2/n$, where $n \in \mathbb N_{\ge 3}$ denotes the degrees of freedom of the molecules. In the meantime this result was extended to every $\gamma\ge 1$. We will restrict ourselves to the case $\gamma\in(1,3)$, which contains the cases of air and the shallow water equations. Since the arguments of compensated compactness are local, we can apply a result by Lions, Perthame and Souganidis \cite{LPS1996} separately to every single pipeline.\medskip \\
Network models for the isentropic gas equations were addressed by many researchers \cite{BHK2006II,BHK2006,CoGa2006,HeSe2008}. Most of the results are based on the wave front tracking technique proposed by Dafermos \cite{Da1972}. The first step consists of finding solutions to so-called generalized Riemann problems at the junctions. These Riemann problems can be used to construct solutions to Cauchy problems if the total variation of the initial data is sufficiently small. Notice, that this is a strong restriction to the technique. Furthermore, the front tracking method is not able to handle networks with arbitrary many junctions which may contain circles.\\
There are also some publications which use a kinetic approach to derive coupling conditions for the macroscopic model \cite{BoKl2018Linear,BoKl2018Scalar,BKP2016,HeMo2009}.
Recently, Borsche and Klar studied half-Riemann problems for scalar \cite{BoKl2018Scalar} and linear \cite{BoKl2018Linear} equations with a kinetic approach to derive macroscopic coupling conditions. Their coupling conditions are defined in a more explicit way compared to our conditions, and they are more interested in numerical aspects. On the other hand, coupling conditions introduced by entropy flux inequalities seem to be the more natural choice for analytical considerations.\medskip\\
The most important problem in studying networks is to define the (physically correct) coupling condition. In the case of $BV$-solutions, the trace of the variables $\rho$ and $u$ always exists, and we can simply give explicit conditions for these traces. A natural condition is conservation of mass or equivalently that the mass-in-flux is equal to the mass-out-flux at the junction. One can simply check that this condition is not sufficient to ensure uniqueness of the solution. The most common additional conditions are equality of pressure, momentum flux or the Bernoulli invariant at the junction. As proven by Reigstad \cite{Re2015}, the first two coupling constants do not produce physically correct solutions in the sense that energy could increase at the junction. Equality of the Bernoulli invariants ensures this property, but this condition is not able to explain the Bernoulli principle. Furthermore, all these macroscopic coupling conditions are not able to describe different geometries of the junction.\\
Next, we explain our approach to construct physically correct coupling conditions. First, notice that we cannot ensure existence of boundary traces of $\rho$ and $u$ itself since we consider $L^\infty$-solutions. 
A similar problem appears if one considers initial boundary value problems. Since the seminal paper by Dubois and LeFloch \cite{DuLe1988}, it is a standard approach to define boundary conditions by inequalities for certain entropy fluxes at the boundary. Existence of solutions with these boundary conditions was proven in \cite{BeBo2002Boundary} for the isentropic gas equations. This result motivates to adapt this idea to networks and illustrates why we want to express the coupling condition in terms of entropy flux traces. The conditions are inherited from the coupling condition on the kinetic level. We couple the kinetic BGK solutions by a certain coupling function $\Psi$, which satisfies inequalities for increasing functions of the kinetic entropy flux traces. As for the Godunov scheme \cite{KSX1997}, we can show that the entropy flux traces are lower semi-continuous with respect to the limit $\epsilon\to 0$. Therefore, the entropy flux inequalities remain true for the macroscopic limit.\\
Our main existence result for the macroscopic solution holds for a large class of kinetic coupling functions $\Psi$ with controlled mass and entropy production. This generality can be used to model the geometry and the local behavior of the junction. In particular, we expect that there is no unique physically correct coupling condition. A similar phenomenon appears in the theory of non-conservative products \cite{DLM1995} which can be used to model gas pipelines with discontinuous cross-sectional area \cite{LeTh2003}. We conjecture that a sufficiently large set of entropy flux inequalities at the junction leads to (in some sense) unique solutions. \\
We give some examples for coupling functions and prove corresponding entropy flux inequalities. For example coupling functions given by a convolution operator or given by linear combinations of the incoming data with the same velocity. Furthermore, we get results for solid wall boundary conditions and pipelines with discontinuous cross-sectional area since they are special cases of our setting.\medskip\\
The paper is organized as follows. In the first part, we use very general coupling conditions to prove the main results in Section \ref{sec:MainResults}. In Section \ref{sec:KineticModel}, we introduce the kinetic model and all necessary properties of it. In Section \ref{sec:SolutionToTheBGKModel}, we prove existence for the coupled kinetic BGK equation. In Section \ref{sec:MaximumPrinciple}, we give a maximum principle on the Riemann invariants which is used to justify the limit $\epsilon \to 0$ and to prove the macroscopic boundary conditions in Section \ref{sec:RelaxationtotheMacroscopicLimit}. This finishes the proofs of the main results, and we continue with some examples for coupling functions and prove entropy flux inequalities in Section \ref{sec:Examples}. In Section \ref{sec:ExtensionsandOutlook}, we show how to generalize our results to networks with arbitrary many junctions and give a short outlook for further research.\medskip\\
We finish the introduction with some notation. The natural space to consider kinetic boundary traces is $L^1_\mu$ with the measure $\mathrm d \mu=|\xi|\mathrm d\xi \mathrm d t$. Sometimes we consider locally integrable functions in $x$ in the sense that $f\in L^1((0,\infty)_t\times \Omega_x\times \mathbb R_\xi)$ for every compact set $\Omega\subset (0,\infty)$ and use the simpler notation $f\in L^1((0,\infty)_t\times (0,\infty)_{\mathrm{loc},x}\times \mathbb R_\xi)$. We write $f\in L^1(\Omega)$ for both $f\in L^1(\Omega,\mathbb R)$ and $f\in L^1(\Omega,\mathbb R^2)$. For $f\in L^1(\mathbb R_\xi,\mathbb R^2)$ with $f\in D_\xi^i$ for a.e. $\xi\in\mathbb R$, we write $f\in L^1(\mathbb R_\xi,D^i_\xi)$. Furthermore, we use combinations or small extensions of these notations.
\section{Main Results}
\label{sec:MainResults}
We study gas networks consisting of $d\in\mathbb N$ (infinitely long) pipelines connected by a single junction. Each pipeline is modeled by a one-dimensional half-space solution to the isentropic gas equations
\begin{align}\label{eq:MacroscopicEquation}
\begin{cases}
\partial_t\rho^i+\partial_x(\rho u)^i &=0,\\
\partial_t(\rho u)^i+\partial_x(\rho u^2+\kappa \rho^\gamma)^i &=0,
\end{cases}
\quad\text{ for }t>0,\,x>0, 
\end{align}
with pressure $\rho^i(t,x)\ge 0$, flow velocity $u^i(t,x)\in \mathbb R$ and $\kappa>0,\, 1<\gamma<3$. The cross-section of the $i$-th pipeline is given by $A^i>0$. Bouchut \cite{Bo1999} introduced a semi-kinetic BGK model for the isentropic gas equations given by
\begin{equation}\label{eq:KineticEquation}
\partial_tf^i+\xi\partial_x f^i=\frac{M[f^i]-f^i}{\epsilon},\quad \text{for }t>0,\,x>0,\,\xi\in\mathbb R,
\end{equation}
where $f^i=f^i(t,x,\xi)\in\mathbb R^2$. $M$ is a vector-valued Maxwellian for this system and will be defined later. We ask for solutions to the BGK model satisfying
\begin{equation}
f^i(t,x,\xi)\in D=\{(f_0,f_1)\in \mathbb R^2|f_0>0 \text{ or }f_0=f_1=0\},
\end{equation}
with initial data
\begin{equation}
f^i(0,x,\xi)=f^{0,i}(x,\xi),\quad x>0,\,\xi\in\mathbb R,
\end{equation}
and coupling condition
\begin{align}\label{eq:CouplingCondition}
f^i(t,0,\xi)=\Psi^i[t,f(t,0,\cdot)](\xi),\quad t>0,\,\xi>0.
\end{align}
The coupling function is given by
\begin{align}
\begin{split}
\Psi\colon (0,\infty)\times L^1_\mu((-\infty,0)_{\xi},D)^d&\to  L^1_\mu((0,\infty)_{\xi},D)^d;\\
(t,g)&\mapsto \Psi[t,g],
\end{split}\label{eq:DefPsi}
\end{align}
and satisfies the continuity property:
\begin{equation}
\begin{split}
L_\mu^1((0,\infty)_{\mathrm{loc},t}\times(-\infty,0)_\xi,D)^d &\to L^1_\mu((0,\infty)_{\mathrm{loc},t}\times (0,\infty)_\xi,D)^d;\\
g&\mapsto \big((t,\xi)\mapsto\Psi[t,g(t,\cdot)](\xi)\big)\qquad \text{ is continuous.}
\end{split}\label{eq:ContPsi}
\end{equation}
\begin{theorem}\label{thm:ExistenceKineticModel}
Assume that $f^{0}\in L^1((0,\infty)_x\times \mathbb R_\xi,D)^d$ and
\begin{equation}\label{eq:FiniteEnergy}
\sum_{i=1}^d A^i \iint_{(0,\infty)\times \mathbb R}H(f^{0,i}(x,\xi),\xi)\;\mathrm{d}x\mathrm{d}\xi\,<\infty.
\end{equation}
Let $\Psi$ satisfy (\ref{eq:DefPsi} -- \ref{eq:ContPsi}). Assume that there exist $b_0,b_H\in L^1((0,\infty)_{\mathrm{loc},t},[0,\infty))$ such that for a.e. $t\in(0,\infty)$ 
\begin{align}
\sum_{i=1}^d A^i\int_0^\infty |\xi|\ \Psi^i_0[t,g](\xi)\;\mathrm{d}\xi\,\label{eq:PsiMassConservation}
&\le\sum_{i=1}^d A^i\int_{-\infty}^0|\xi|\ g^i_0(\xi)\;\mathrm{d}\xi +b_0(t),\\
\sum_{i=1}^d A^i\int_0^\infty |\xi|\ H(\Psi^i[t,g](\xi),\xi)\;\mathrm{d}\xi\,\label{eq:PsiEntropyConservation}
&\le\sum_{i=1}^d A^i \int_{-\infty}^0|\xi|\ H(g^i(\xi),\xi)\;\mathrm{d}\xi\,+b_H(t),
\end{align}
for all $g\in L^1_\mu((-\infty,0)_\xi,D)^d$. The function $H\colon D\times\mathbb R\to[0,\infty)$ is the kinetic energy and will be defined in (\ref{eq:defkineticenergy}).
Then, there exists a solution $f=(f^1,...,f^d)$ to (\ref{eq:KineticEquation} -- \ref{eq:CouplingCondition}) satisfying
\begin{gather}
f^i\in C([0,\infty)_t,L^1((0,\infty)_x\times \mathbb R_\xi))\cap C([0,\infty)_x,L_\mu^1((0,\infty)_{\mathrm{loc},t}\times \mathbb R_\xi)),\\
\text{ for any } t\ge 0, f^i(t,x,\xi)\in D\text{ a.e. in }(0,\infty)_x\times \mathbb R_\xi,\\
H(f^i(t,x,\xi),\xi)\in L^\infty((0,\infty)_t,L^1((0,\infty)_x\times\mathbb R_\xi)),\\
\partial_t\left(\int_{\mathbb R}f^i\;\text{d}\xi\right)+\partial_x\left(\int_{\mathbb R}\xi f^i\;\text{d}\xi\right)=0. \label{eq:DifferentialEquationMoments}
\end{gather}
Furthermore, we have for any $t\in [0,\infty)$
\begin{align}
\sum_{i=1}^d A^i\iint_{(0,\infty)\times \mathbb R} f^i_0(t,x,\xi)\;\mathrm{d}x\mathrm{d}\xi
&\le\sum_{i=1}^d A^i\iint_{(0,\infty)\times \mathbb R} f^{0,i}_0(x,\xi)\;\mathrm{d}x\mathrm{d}\xi+\int_0^t b_0(s)\;\text{d}s\label{eq:SolutionMassConservation}\\
\sum_{i=1}^d A^i\iint_{(0,\infty)\times \mathbb R} H(f^i(t,x,\xi),\xi)\;\mathrm{d}x\mathrm{d}\xi
&\le\sum_{i=1}^d A^i\iint_{(0,\infty)\times \mathbb R} H(f^{0,i}(x,\xi),\xi)\;\mathrm{d}x\mathrm{d}\xi+\int_0^t b_H(s)\;\text{d}s.\label{eq:SolutionEntropyConservation}
\end{align}
If we additionally assume equality in (\ref{eq:PsiMassConservation}) for a.e. $t\in(0,\infty)$, we obtain equality in (\ref{eq:SolutionMassConservation}).
\end{theorem} 
\begin{remark}
In (\ref{eq:PsiMassConservation}), the function $b_0$ controls the local mass production at the junction. More precisely, the mass leaving the junction is bounded by the mass entering the junction plus the bound on the mass production $b_0$. In the physically relevant case, we expect $b_0=0$ since this implies that no mass is produced at the junction. We use this local estimate to prove the global mass estimate in (\ref{eq:SolutionMassConservation}). Similarly, we use (\ref{eq:PsiEntropyConservation}) with the kinetic energy functional $H$ to obtain a global estimate on the kinetic energy in (\ref{eq:SolutionEntropyConservation}).\\
Notice that we could use similar assumptions with other (symmetric) kinetic entropy functions to obtain similar global bounds on the kinetic entropy in the network. In particular, we get equality in (\ref{eq:SolutionMassConservation}) if we assume equality in (\ref{eq:PsiMassConservation}) and use $H_{-\mathbbm 1}(f,\xi)=-f_0$.
\end{remark}
In the next step, we want to take the limit $\epsilon\to 0$ to obtain a macroscopic solution to the isentropic gas equations (\ref{eq:MacroscopicEquation}). As usual, we ask for an entropy solution to (\ref{eq:MacroscopicEquation}) which additionally satisfies 
\begin{equation}\label{eq:EntropySolution}
\partial_t(\eta_S(\rho^i,u^i))+\partial_x(G_S(\rho^i,u^i))\le 0 \quad\text{in }(0,\infty)_t\times(0,\infty)_x
\end{equation}
for entropy pairs $(\eta_S,G_S)$ parametrized by a convex function $S\colon \mathbb R\to\mathbb R$ of class $C^1$ with $|S(v)|\le B(1+v^2)$ for a constant $B>0$.\\
To justify the limit, we will need uniform $L^\infty$-bounds on the solutions which can be obtained by a maximum principle for the (kinetic) Riemann invariants. We introduce the family of kinetic invariant domains ($\tilde D_\xi^1,\dots,\tilde D_\xi^d)$ by
\begin{equation}
\tilde D_\xi^i=\{f\in D;\,f=0 \text{ or } \omega^i_{\min}\le\omega_1(f,\xi)\le\omega_2(f,\xi)\le \omega^i_{\max}\}.
\end{equation}
We assume $f^{0,i}\in L^1((0,\infty)_x\times\mathbb R_\xi,\tilde D_\xi^i)$ and for a.e. $t\in(0,\infty)$ 
\begin{gather}\label{eq:KineticInvariancePsi}
\begin{split}
\sum_{i=1}^d A^i\int_0^\infty |\xi|\ H_{S_\omega^i}(\Psi^i[t,g](\xi),\xi)\;\mathrm{d}\xi\,
\le\sum_{i=1}^d A^i \int_{-\infty}^0|\xi|\ H_{S_\omega^i}(g^i(\xi),\xi)\;\mathrm{d}\xi,\\
\text{for all $g\in L^1_\mu((-\infty,0)_\xi,D)^d$, where $S_\omega^i(v)=(v-\omega_{\max}^i)_+^2+(\omega_{\min}^i-v)_+^2$}.
\end{split}
\end{gather}
This assumption implies $f\in \tilde D_\xi^i$ a.e. $t,x,\xi$ and leads to the uniform $L^\infty$-bounds (see Theorem \ref{thm:MaximumPrinciple}).
\begin{theorem}\label{thm:ExistenceMacroscopicEquation}
Let $f_\epsilon$ be the solution obtained in Theorem \ref{thm:ExistenceKineticModel} with initial data $f^{0,i}\in L^1((0,\infty)_x\times \mathbb R_\xi,\tilde D_\xi^i)$ satisfying (\ref{eq:FiniteEnergy}) and coupling function $\Psi$ satisfying (\ref{eq:DefPsi}), (\ref{eq:ContPsi}) and (\ref{eq:KineticInvariancePsi}) for some $-\infty<\omega_{\min}^i<\omega_{\max}^i<\infty$. Then $(\rho_\epsilon^i,\rho^i_\epsilon u^i_\epsilon)(t,x)=\int_\mathbb R f^i_{\epsilon}(t,x,\xi) \text{d}\xi$ are uniformly bounded in $L^\infty((0,\infty)_t\times (0,\infty)_x)$. After passing if necessary to a subsequence, $(\rho_\epsilon^i,\rho_\epsilon^i u_\epsilon^i)$ converge a.e. in $(0,\infty)_t\times (0,\infty)_x$ to an entropy solution $(\rho^i,\rho^i u^i)$ to (\ref{eq:MacroscopicEquation}), (\ref{eq:EntropySolution}) remaining in $\tilde D^i$ with initial data $(\rho^{0,i},\rho^{0,i} u^{0,i})=\int_\mathbb R f^{0,i}\text{d}\xi$. Furthermore, after passing if necessary to a subsequence again,
\begin{equation}
\overline{G_S(\rho^i,u^i)}(t,0)\le \psi_S^i(t):=\underset{\epsilon\to 0}{\operatorname{w*-lim}}\int_\mathbb R \xi\,H_S(f_\epsilon(t,0,\xi),\xi)\;\text{d}\xi
\end{equation}
a.e. $t>0$, where $S\colon \mathbb R\to\mathbb R$ is convex, of class $C^1$ and $|S(v)|\le B(1+v^2)$ for a constant $B$. 
In particular, $\overline{G_S(\rho^i,u^i)}(t,0)$ and $\psi_S^i(t)$ are bounded in $L^\infty_t(0,\infty)$.
\end{theorem}
\begin{corollary}\label{cor:MonotonicityGammaLimit}
Let all assumptions in Theorem \ref{thm:ExistenceMacroscopicEquation} be satisfied. Let $p\in \mathbb N$, $i_l\in\{1,\dots,d\}$, $S_l\colon\mathbb R\to\mathbb R$ convex, of class $C^1$ with $|S_l(v)|\le B_l(1+v^2)$, $l=1,\dots,p$.
Let $\Gamma\colon (0,\infty)_t\times \mathbb R^p\to \mathbb R$ be such that $\Gamma[t,\cdot]$ is uniformly bounded on compact sets and increasing in every argument with $S_l\notin \operatorname{span}\{\mathbbm 1,v\}$. Then, 
\begin{equation*}\label{eq:DefboundbS}
\Gamma[t,\overline{G_{S_1}(\rho^{i_1},u^{i_1})}(t,0),\dots,\overline{G_{S_p}(\rho^{i_p},u^{i_p})}(t,0)]\le \Gamma[t,\psi_{S_1}^{i_1}(t),\dots,\psi_{S_p}^{i_p}(t)]\le b_{\Gamma, \mathcal S}(t) \quad \text{a.e. }t>0,
\end{equation*}
where $b_{\Gamma, \mathcal S}\in L^\infty_t(0,\infty)$ depend only on $\Psi$.
\end{corollary}
\section{Basic Properties of the BGK Model}
\label{sec:KineticModel}
In this section, we recall several properties of the BGK model for isentropic gas. The section is based on \cite{BeBo2000,BeBo2002Kinetic,BeBo2002Boundary} and all proofs are given there. Almost all results in this section are point-wise or independent of the coupling condition. Therefore, we restrict ourselves to the case $d=1$ and omit the indices. The Maxwellian is given by
\begin{equation}
M[f](t,x,\xi)=M(\rho(t,x),u(t,x),\xi)
\end{equation}
with
\begin{equation}\label{eq:DefinitionrhouMoments}
\rho(t,x)=\int_\mathbb R f_0(t,x,\xi)\;\text{d}\xi\,,\quad \rho(t,x)u(t,x)=\int_\mathbb R f_1(t,x,\xi)\;\text{d}\xi\,
\end{equation}
and
\begin{gather}
M(\rho,u,\xi)=(\chi(\rho,\xi-u),((1-\theta)u+\theta \xi)\chi(\rho,\xi-u)),\\
\chi(\rho,\xi)=c_{\gamma,\kappa}(a_\gamma^2\rho^{\gamma-1}-\xi^2)_+^\lambda,\\
\theta=\frac{\gamma-1}{2},\quad\lambda=\frac{1}{\gamma-1}-\frac{1}{2},\quad c_{\gamma,\kappa}=\frac{a_\gamma^{-2/(\gamma-1)}}{J_\lambda},\\
J_\lambda=\int_{-1}^1(1-z^2)^\lambda\;\text{d}z\,=\frac{\sqrt{\pi}\Gamma(\lambda+1)}{\Gamma(\lambda+3/2)},\quad a_\gamma=\frac{2\sqrt{\gamma\kappa}}{\gamma-1}.
\end{gather}
The Maxwellian satisfies the following moment properties
\begin{align}
\int_\mathbb R M(\rho,u,\xi)\;\text{d}\xi&=(\rho,\rho u),\label{eq:0thMomentMaxwellian}\\
\int_\mathbb R \xi M(\rho,u,\xi)\;\text{d}\xi&=(\rho u,\rho u^2+\kappa \rho^\gamma)=F(\rho,u),\label{eq:1stMomentMaxwellian}
\end{align}
for every $\rho\ge 0$ and $u\in\mathbb R$. A useful property of the isentropic gas equations is the huge class of entropies parametrized by convex functions $S\colon\mathbb R\to\mathbb R$. The kinetic entropies are defined by 
\begin{align}
H_S(f,\xi)&=\int_\mathbb R \Phi(\rho(f,\xi),u(f,\xi),\xi,v)S(v)\;\text{d}v\quad \text{for }f\neq 0,\\
H_S(0,\xi)&=0,
\end{align}
where
\begin{align}
u(f,\xi)&=\frac{f_1/f_0-\theta \xi}{1-\theta},\\
\rho(f,\xi)&=a_\gamma^{-\tfrac{2}{\gamma-1}}\left(\left(\frac{f_1/f_0-\xi}{1-\theta}\right)^2+\left(\frac{f_0}{c_{\gamma,\kappa}}\right)^{1/\lambda}\right)^{\tfrac{1}{\gamma-1}},
\end{align}
is the inverse relation to $f=M(\rho,u,\xi)$. The kernel $\Phi$ is defined by
\begin{gather}
\Phi(\rho,u,\xi,v)=\frac{(1-\theta)^2}{\theta}\frac{c_{\gamma,\kappa}}{J_\lambda}\mathbbm{1}_{\omega_1<\xi<\omega_2}\mathbbm{1}_{\omega_1<v<\omega_2}|\xi-v|^{2\lambda-1}\Upsilon_{\lambda-1}(z),\\
z=\frac{(\xi+v)(\omega_1+\omega_2)-2(\omega_1\omega_2+\xi v)}{(\omega_2-\omega_1)|\xi-v|},\\
\Upsilon_{\lambda-1}(z)=\int_1^z (y^2-1)^{\lambda -1}\;\text{d}y,\quad z\ge 1.
\end{gather}
$\Phi$ is symmetric in $\xi,v$, satisfies $\Phi\ge 0$ and $\int_{\mathbb R}(1,v)\Phi(\rho,u,\xi,v)\;\text{d}v=M(\rho,u,\xi)$. The macroscopic entropy and entropy flux are given by
\begin{align}
\eta_S(\rho,u)&=\int_{\mathbb R}\chi(\rho,v-u) S(v)\;\text{d}v=\int_{\mathbb R} H_S(M(\rho,u,\xi),\xi)\;\text{d}\xi,\\
G_S(\rho,u)&=\int_\mathbb R [(1-\theta)u+\theta v]\ \chi(\rho,v-u)S(v)\;\text{d}v\\
&=\int_\mathbb R\xi H_S(M(\rho,u,\xi),\xi)\;\text{d}\xi.
\end{align}
The kinetic entropy parametrized by $S(v)=v^2/2$ is given by
\begin{equation}\label{eq:defkineticenergy}
H(f,\xi)=\frac{\theta}{1-\theta}\frac{\xi^2}{2} f_0+\frac{\theta}{2c_{\gamma,\kappa}^{1/\lambda}}\frac{f_0^{1+1/\lambda}}{1+1/\lambda}
+\frac{1}{1-\theta}\frac{1}{2}\frac{f_1^2}{f_0}-\frac{\theta}{1-\theta}\xi f_1,
\end{equation}
and the corresponding macroscopic entropy is the physical energy
\begin{align}
\eta(\rho,u)=\frac{\rho u^2}{2}+\frac{\kappa}{\gamma-1}\rho^\gamma,\quad G(\rho,u)=\frac{\rho u^3}{2}+\frac{\gamma\kappa}{\gamma-1}\rho^\gamma u.
\end{align}
The isentropic gas equations admit the Riemann invariants 
\begin{align}
\omega_1&=u-a_\gamma\rho^\theta,\quad\omega_2=u+a_\gamma\rho^\theta,
\end{align}
for $\rho\neq 0$. A kinetic version of them is given by
\begin{align}
\omega_1&=u(f,\xi)-a_\gamma\rho(f,\xi)^\theta,\quad\omega_2=u(f,\xi)+a_\gamma\rho(f,\xi)^\theta,
\end{align}
for $f\neq 0$. We recall several properties of the previous definitions:
\begin{lemma}[{\cite[Lemma 3.1]{BeBo2002Kinetic}}]\label{lemma:Bijectionfandomega}
The sets $\{f_0>0\}$ and $\{\omega_1<\xi<\omega_2\}$ are in bijection by the functions
\begin{equation}
Q(f)=
\begin{pmatrix}
\frac{f_1/f_0-\theta \xi}{1-\theta}-\sqrt{\left(\frac{f_1/f_0-\xi}{1-\theta}\right)^2+\left(\frac{f_0}{c_{\gamma,\kappa}}\right)^{1/\lambda}}\\
\frac{f_1/f_0-\theta \xi}{1-\theta}+\sqrt{\left(\frac{f_1/f_0-\xi}{1-\theta}\right)^2+\left(\frac{f_0}{c_{\gamma,\kappa}}\right)^{1/\lambda}}
\end{pmatrix}
\end{equation}
and 
\begin{equation}
R(\omega)=
\begin{pmatrix}
c_{\gamma,\kappa}\left(\left(\frac{\omega_2-\omega_1}{2}\right)^2-\left(\xi-\frac{\omega_1+\omega_2}{2}\right)^{2}\right)_+^\lambda\\
\left((1-\theta)\frac{\omega_1+\omega_2}{2}+\theta \xi\right) c_{\gamma,\kappa}\left(\left(\frac{\omega_2-\omega_1}{2}\right)^2-\left(\xi-\frac{\omega_1+\omega_2}{2}\right)^{2}\right)_+^\lambda	
\end{pmatrix}.
\end{equation}
\end{lemma}
\begin{proposition}[{\cite[Lemma 3.2, Proposition 3.3, Corollary 3.4]{BeBo2002Kinetic}}]\label{prop:SmoothnessPropertiesHandS}
\begin{enumerate}[label=(\roman*)]
	\item If $S\colon \mathbb R\to \mathbb R$ is of class $C^k$, then the functions $(\rho,u)\mapsto \eta_S(\rho,u)$ and $(\rho,q)\mapsto\eta_S(\rho,u)$ with $q=\rho u$ are $C^k$ in $\{\rho>0\}$.
	\item If $S\colon \mathbb R\to \mathbb R$ is of class $C^k$, then $H_S(\cdot,\xi)$ is $C^k$ in $\{f_0>0\}$.
	\item If $S\colon\mathbb R\to\mathbb R$ is bounded on compact sets, then the function $(\omega_1,\omega_2)\mapsto G_S(\rho,u,\xi):=H_S(M(\rho,u,\xi),\xi)$ is continuous differentiable in $\{\omega_1<\xi<\omega_2\}$ with 
	\begin{equation*}
		\frac{\partial G_S}{\partial \omega_i}(\rho,u,\xi)=\int_{\mathbb R}\frac{\partial\Phi}{\partial \omega_i}(\rho,u,\xi,v)\,S(v)\;\mathrm{d}v,\quad\text{for }i=1,2. 
	\end{equation*}
	\item If $S\colon\mathbb R\to\mathbb R$ is bounded on compact sets, then $H_S(\cdot,\xi)$ is continuous at $0$ in $\{f\in D; |f_1|\le Af_0\}$, for any $A>0$.
	\item If $S\colon\mathbb R\to\mathbb R$ is of class $C^1$, then we have $H_S '(M(\rho,u,\xi),\xi)=\eta_S'(\rho,u)$ whenever $M(\rho,u,\xi)_0>0$.
\end{enumerate}
\end{proposition}
\begin{proposition}[{\cite[Proposition 3.5]{BeBo2002Kinetic}}]\label{prop:ConvexityofSandH}
\begin{enumerate}[label=(\roman*)]
	\item If $S\colon\mathbb R\to\mathbb R$ is convex and of class $C^2$, then $\eta_S$ is convex in $\{\rho>0\}$ and if $S''>0$, then $\eta_S''>0$.
	\item If $S\colon \mathbb R\to\mathbb R$ is convex, then $H_S(\cdot,\xi)$ is convex in $D$. 
\end{enumerate}
\end{proposition}
\begin{lemma}[{\cite[Lemma 2.3]{BeBo2000}}]\label{lemma:EstimatesHandf}
There exist $\epsilon_0,\epsilon_1>0$ such that for any $f\in D, \xi\in\mathbb R$, we have
\begin{equation}
H(f,\xi)\ge \epsilon_0 f_0^{p_0}+\epsilon_1 |f_1|^{p_1},
\end{equation}
with
\begin{equation}
p_0=1+1/\lambda>1,\quad p_1=2(1+\lambda)/(1+2\lambda)>1.
\end{equation}
Furthermore,
\begin{align*}
|f_1|\le \sqrt{2H(f,\xi)\,f_0}.
\end{align*}
\end{lemma}
\begin{proposition}[Subdifferential inequality, {\cite[Proposition 4.1]{BeBo2002Kinetic}}]\label{prop:SubdifferentialInequality}
If $S\colon\mathbb R\to\mathbb R$ is convex, of class $C^1$, then for every $f\in D,\,\rho\ge 0$ and $u,\xi\in \mathbb R$, we have
\begin{equation}
H_S(f,\xi)\ge H_S(M(\rho,u,\xi),\xi)+T_S(\rho,u)(f-M(\rho,u,\xi)),\\
\end{equation}
with
\begin{equation}
T_S(\rho,u)=\frac{1}{J_\lambda}\int_{-1}^1 (1-z^2)^\lambda 
\begin{pmatrix}
    S(u+a_\gamma \rho^\theta z)+(\theta a_\gamma \rho^\theta z-u)S'(u+a_\gamma \rho^\theta z)\\
		S'(u+a_\gamma \rho^\theta z)
\end{pmatrix}
\mathrm{d}z,
\end{equation}
which coincides with $\eta'_S(\rho,u)$ for $\rho>0$. If $f\neq 0$, we have
\begin{equation}
(H'_S(f,\xi)-T_S(\rho,u))(M(\rho,u,\xi)-f)\le 0.
\end{equation}
\end{proposition}
\begin{corollary}[Entropy minimization principle, {\cite[Corollary 4.4]{BeBo2002Kinetic}}]\label{cor:EntropyMinimizationPrinciple}
Assume that $S\colon \mathbb R\to\mathbb R$ is convex, of class $C^1$ and such that $|S(v)|\le B(1+v^2)$ for some $B\ge 0$. Consider $f\in L^1(\mathbb R_\xi)$ such that $f\in D$ a.e. and $\int_{\mathbb R} H(f(\xi),\xi)\;\mathrm{d}\xi<\infty$. Then, $H_S(f(\xi),\xi)$ and $H_S(M[f](\xi),\xi)$ lie in $L^1(\mathbb R_\xi)$ with
\begin{equation}
\int_{\mathbb R} H_S(M[f](\xi),\xi)\;\mathrm{d}\xi\le \int_{\mathbb R}H_S(f(\xi),\xi)\;\mathrm{d}\xi.
\end{equation}
\end{corollary}
\section{Solution to the BGK Model}
\label{sec:SolutionToTheBGKModel}
In this section, we prove Theorem \ref{thm:ExistenceKineticModel} by adapting the arguments in \cite{BeBo2000}.
\begin{lemma}\label{lemma:SolutionLinearProblem}
Let $h\in L^1((0,T)_t,L^1((0,\infty)_x\times \mathbb R_\xi))^d,f^{0}\in L^1((0,\infty)_x\times \mathbb R_\xi)^d$ and 
\begin{equation*}
\Psi\colon (0,T)\times L^1_\mu((-\infty,0)_\xi)^d\to  L^1_\mu((0,\infty)_\xi)^d.
\end{equation*}
Then there exists a unique solution
\begin{equation}
f^i\in C([0,T]_t,L^1((0,\infty)_x\times\mathbb R_\xi))\cap C([0,\infty)_x,L_\mu^1((0,T)_t\times\mathbb R_\xi))
\end{equation}
to the problem
\begin{equation}\label{eq:SolutionLinearProblem}
\begin{cases}
\partial_t f^i+\xi\partial_x f^i=\frac{h^i-f^i}{\epsilon},&t\in(0,T),x>0,\xi\in\mathbb R,\\
f^i(0,x,\xi)=f^{0,i}(x,\xi),&x>0,\xi\in\mathbb R,\\
f^i(t,0,\xi)= \Psi^i[t,f(t,0,\cdot)](\xi),&t\in(0,T),\xi>0,
\end{cases}
\end{equation}
for $i=1,...,d$. Furthermore, for any $t\in[0,T]$, a.e. $x>0$, $\xi\in\mathbb R$,
\begin{align}
f^i&(t,x,\xi)=\left[f^{0,i}(x-t\xi,\xi)e^{-t/\epsilon}+\frac{1}{\epsilon}\intt{e^{-s/\epsilon}h^i(t-s,x-s\xi,\xi)} \right]_{x>t\xi}\nonumber\\
&+\bigg[ \Psi^i[t-x/\xi,f(t-x/\xi,0,\cdot)](\xi) e^{-x/(\epsilon\xi)}+\frac{1}{\epsilon}\int_0^{x/\xi}e^{-s/\epsilon}h^i(t-s,x-s\xi,\xi)\;\text{d}s\, \bigg]_{x<t\xi},\label{eq:CharacteristicFormula}
\end{align}
and
\begin{equation}
\lVert{f^i}\rVert_{C_x([0,\infty),L_\mu^1((0,T)_t\times\mathbb R_\xi))}\le \lVert{f^{i,0}}\rVert_{L^1}+\lVert{\Psi^i[\cdot,f(x=0)]}\rVert_{L^1_\mu}+\frac{1}{\epsilon}\lVert{h^i}\rVert_{L^1}.
\end{equation}
\end{lemma}
\begin{proposition}\label{prop:EstimaresHundfglobal}
Let $f\in L^1((0,\infty)_x\times\mathbb R_\xi,D)^d$ be such that
\begin{gather*}
\sum_{i=1}^d A^i \iint_{(0,\infty)\times\mathbb R}H(f^i(x,\xi),\xi)\;\mathrm{d}x\mathrm{d}\xi \le C_H,\\
\sum_{i=1}^d A^i \iint_{(0,\infty)\times\mathbb R}f^i_0(x,\xi)\;\mathrm{d}x\mathrm{d}\xi \le C_0.
\end{gather*}
Then, we have 
\begin{gather}
\sum_{i=1}^d A^i \iint_{(0,\infty)\times\mathbb R}|f^i_1(x,\xi)|\;\mathrm{d}x\mathrm{d}\xi \le \sqrt{2C_0 C_H},\\
\sum_{i=1}^d A^i \iint_{(0,\infty)\times\mathbb R}\xi^2 f^i_0(x,\xi)\;\mathrm{d}x\mathrm{d}\xi \le \frac{4}{\theta} C_H,\\
\sum_{i=1}^d A^i \iint_{(0,\infty)\times\mathbb R}|\xi|\,|f^i_1(x,\xi)|\;\mathrm{d}x\mathrm{d}\xi \le \sqrt{\frac{8}{\theta}} C_H,\\
f^i_k \text{ is bounded in } L^{p_k}((0,\infty)_x\times \mathbb R_\xi) \text{ for }i=1,...,d \text{ and } k=0,1.
\end{gather}
\end{proposition}
\begin{proposition}\label{prop:Stability}
Let $f^0\in L^1((0,\infty)_x\times\mathbb R_\xi,D)^d$ and let $\Psi$ be as in (\ref{eq:DefPsi} -- \ref{eq:ContPsi}). Let $C_0,C_H\in L^\infty_t(0,T)$ and $g,g_n\in L^\infty((0,T)_t,L^1((0,\infty)_x\times \mathbb R_\xi,D))^d$ such that
\begin{align*}
\sum_{i=1}^d A^i \iint_{(0,\infty)\times \mathbb R}H(g^i(t,x,\xi),\xi)\;\mathrm{d}x\mathrm{d}\xi&\le C_H(t),& \sum_{i=1}^d A^i \iint_{(0,\infty)\times \mathbb R}(g^i)_0(t,x,\xi)\;\mathrm{d}x\mathrm{d}\xi&\le C_0(t),\\
\sum_{i=1}^d A^i \iint_{(0,\infty)\times \mathbb R}H(g^i_n(t,x,\xi),\xi)\;\mathrm{d}x\mathrm{d}\xi &\le C_H(t),& \sum_{i=1}^d A^i \iint_{(0,\infty)\times \mathbb R}(g^i_n)_0(t,x,\xi)\;\mathrm{d}x\mathrm{d}\xi &\le C_0(t),
\end{align*}
for a.e. $t\in (0,T)$. Set $\rho=(\rho^1,...,\rho^d), \rho u=(\rho^1 u^1,...,\rho^d u^d)$ and
\begin{align*}
(\rho^i(t,x),\rho^i u^i(t,x))&=\int_\mathbb R g^i(t,x,\xi)\;\mathrm{d}\xi,\\
(\rho_n^i(t,x),\rho_n^i u_n^i(t,x))&=\int_\mathbb R g_n^i(t,x,\xi)\;\mathrm{d}\xi.
\end{align*}
If $\rho_n\to \rho$ and $\rho_n u_n\to \rho u$ as $n\to\infty$ in $L^1((0,T)_t\times (0,\infty)_{\mathrm{loc},x})^d$, then there exists a subsequence such that $F(g_n)\rightarrow F(g)$ in $C([0,T]_t,L^1((0,\infty)_{\mathrm{loc},x}\times \mathbb R_\xi))^d$, where $F(g)$ is a solution to (\ref{eq:SolutionLinearProblem}) with $h^i=M[g^i]$. 
\end{proposition}
\begin{proof}
First, we have to check that $F(g)$ is well-defined.
Notice that we are not exactly in the setting of Lemma \ref{lemma:SolutionLinearProblem} since the domain of $\Psi$ is different. Therefore we apply Lemma \ref{lemma:SolutionLinearProblem} with $\tilde \Psi[t,r]:=\Psi[t,\tilde r]$ where $\tilde r^i(\xi)=r^i(\xi)$ if $r^i(\xi)\in D$ and $r^i(\xi)=0$ else. Corollary \ref{cor:EntropyMinimizationPrinciple} and Proposition \ref{prop:EstimaresHundfglobal} imply $M[g^i],M[g^i_n]\in L^\infty((0,T)_t,L^1((0,\infty)_x\times\mathbb R_\xi)$ with uniform bounds. It remains to prove that $F(g)$ is a solution to (\ref{eq:SolutionLinearProblem}) with the coupling function $\Psi$ or equivalently $F^i(g)(t,0,\xi)\in D$ for a.e. $t\in(0,T),\xi<0$. The solution formula is
\begin{equation*}
F^i(g)(t,0,\xi)=f^{0,i}(-t\xi,\xi)e^{-t/\epsilon}+\frac{1}{\epsilon}\int_0^t e^{-s/\epsilon}M[g^i](t-s,-s\xi,\xi)\;\mathrm{d}s,
\end{equation*}
for a.e. $t\in(0,T),\,\xi<0$, which gives $F^i(g)_0(t,0,\xi)\ge 0$ for a.e. $t\in(0,T),\,\xi<0$. Assuming $F^i_0(g)(t,0,\xi)=0$ implies $f_0^{0,i}(-t\xi,\xi)=0$ and $M[g^i]_0(t-s,-s\xi,\xi)=0$ a.e. $s\in(0,t)$, but $f^{0,i}\in D$ and $M[g^i]\in D$ a.e. imply $F(g^i)_1(t,0,\xi)=0$ a.e. $t\in(0,T),\,\xi<0$. We conclude that $F^i(g)(t,0,\xi)\in D$ a.e. $t\in(0,T),\,\xi<0$. The proof for $F(g_n)$ works in the same way.\\
We continue with the stability of $F$. As in \cite{BeBo2000}, we have 
\begin{equation*}
M[g_n^i]\to M[g^i] \quad\text{as }n\to\infty\quad \text{in }L^1((0,T)_t\times (0,\infty)_{\text{loc},x}\times \mathbb R_\xi). 
\end{equation*}
We fix an $t\in[0,T]$ and consider the parts $\{x>t\xi\}$ and $\{x<t\xi\}$ of the domain separately. 
For the domain $\{x>t\xi\}$, we proceed as in \cite{BeBo2000}. We have
\begin{align*}
&\iint_{(0,R)\times (-S,S)}|F^i(g_n)-F^i(g)|(t,x,\xi)\,\mathbbm 1 _{\{x>t\xi\}}\;\text{d}x\text{d}\xi \\
&= \frac{1}{\epsilon}\int_0^t e^{-s/\epsilon}\iint_{(0,R)\times (-S,S)}|M[g_n^i]-M[g^i]|(t-s,x-s\xi,\xi)\,\mathbbm 1_{\{x>t\xi\}}\;\text{d}x\text{d}\xi\,\text{d}s\\
&\le \frac{1}{\epsilon} \lVert M[g_n^i]-M[g^i]\rVert_{L^1((0,T)_t\times (0,R+TS)_x\times (-S,S)_\xi)}\rightarrow 0\quad \text{as }n\to\infty,
\end{align*}
for arbitrary constants $R,S>0$. On the other hand, we have
\begin{align*}
&\iint_{(0,\infty)\times \mathbb R\backslash [-S,S]}|F^i(g_n)-F^i(g)|(t,x,\xi)\,\mathbbm 1_{\{x>t\xi\}}\;\text{d}x\text{d}\xi\\
&\le \frac{1}{\epsilon}\int_0^t e^{-s/\epsilon}\iint_{(0,\infty)\times \mathbb R\backslash [-S,S]}\frac{|\xi|}{S}|M[g_n^i]-M[g^i]|(t-s,x-s\xi,\xi)\,\mathbbm 1_{\{x>t\xi\}}\;\text{d}x\text{d}\xi\,\text{d}s\\
&\le\frac{1}{\epsilon S}\lVert\xi M[g_n^i]-\xi M[g^i]\rVert_{L^1((0,T)\times (0,\infty)_x\times \mathbb R_\xi)}.
\end{align*}
Since Proposition \ref{prop:EstimaresHundfglobal}, the last norm is bounded and we get convergence on the domain $\{x>t\xi\}$. \\
On $\{x<t\xi\}$, we have
\begin{align}
&|F^i(g_n)-F^i(g)|(t,x,\xi) \nonumber \\
&\le |\Psi^i[t-x/\xi,F(g_n)(t-x/\xi,0,\cdot)]- \Psi^i[t-x/\xi,F(g)(t-x/\xi,0,\cdot)]|(\xi)\nonumber\\
&\quad+\frac{1}{\epsilon} \int_0^{x/\xi}|M[g^i_n]-M[g^i]|(t-s,x-s\xi,\xi)\;\text{d}s.\label{eq:stabilityestimatedifferentformorethanonejunction}
\end{align}
The second term on the right hand side can be handled with similar arguments as above. The remaining term is
\begin{align*}
\iint_{(0,\infty)\times(0,\infty)}&|\Psi^i[t-x/\xi,F(g_n)(t-x/\xi,0,\cdot)]- \Psi^i[t-x/\xi,F(g)(t-x/\xi,0,\cdot)]|(\xi)\,\mathbbm 1_{\{x<t\xi\}}\;\text{d}x\text{d}\xi\\
&=\iint_{(0,t)\times(0,\infty)}\xi\, |\Psi^i[s,F(g_n)(s,0,\cdot)]- \Psi^i[s,F(g)(s,0,\cdot)]|(\xi)\;\text{d}s\text{d}\xi,
\end{align*}
but this goes to zero since 
\begin{align*}
&\iint_{(0,T)\times (-\infty,0)} |\xi F^k(g_n)-\xi F^k(g)|(t,0,\xi)\;\text{d}t\text{d}\xi\\
&\le \frac{1}{\epsilon}\iiint_{(0,T)_t\times(0,t)_s\times (-\infty,0)_\xi}|\xi M[g^k_n]-\xi M[g^k]|(t-s,-s\xi,\xi)\;\text{d}t\text{d}s\text{d}\xi\to 0 
\end{align*}
as $n\to \infty \text{ for } k=1,...,d$ and the continuity assumption (\ref{eq:ContPsi}) on $\Psi$.
This completes the convergence proof on $\{x<t\xi\}$ and gives the stability result since the estimates are uniform in $t\in[0,T]$.
\end{proof}
Fix $T>0,\, f^0\in L^1((0,\infty)_x\times \mathbb R_\xi,D)^d,\, \Psi\colon (0,\infty)\times L_\mu^1((-\infty,0),D)^d\to L_\mu^1((0,\infty)_\xi,D)^d$ such that the assumptions in Theorem \ref{thm:ExistenceKineticModel} are satisfied. We set
\begin{align}
C_H(t)&=\sum_{i=1}^d A^i\iint_{(0,\infty)\times\mathbb R}H(f^{0,i}(x,\xi),\xi)\;\text{d}x\text{d}\xi +\int_0^t b_H(s)\;\text{d}s,\\
C_0(t)&=\sum_{i=1}^d A^i\iint_{(0,\infty)\times\mathbb R}f^{0,i}_0(x,\xi)\;\text{d}x\text{d}\xi +\int_0^t b_0(s)\;\text{d}s.
\end{align}
We define the set $C$ by all functions $g\in L^\infty((0,T)_t,L^1((0,\infty)_x\times \mathbb R_\xi))^d$ satisfying (\ref{con:C1} -- \ref{con:C3}) for a.e. $t\in[0,T]$, where
\begin{gather}
g^i(t,x,\xi)\in D \quad \text{a.e. in }(0,\infty)_x\times \mathbb R_\xi, \label{con:C1}\tag{C1}\\
\sum_{i=1}^d A^i \iint_{(0,\infty)\times \mathbb R} H(g^i(t,x,\xi),\xi)\;\text{d}x\text{d}\xi\le C_H(t),\label{con:C2}\tag{C2}\\
\sum_{i=1}^d A^i \iint_{(0,\infty)\times \mathbb R} g_0^i(t,x,\xi)\;\text{d}x\text{d}\xi\le C_0(t).\label{con:C3}\tag{C3}
\end{gather}
Let us also introduce
\begin{align*}
\tilde C&=\Big\{g\in C([0,T]_t,L^1((0,\infty)_x\times\mathbb R_\xi))^d \text{ satisfying (\ref{con:C4})}\\
&\qquad\qquad \qquad \qquad \qquad \text{ and (\ref{con:C1} -- \ref{con:C3}) for all }t\in[0,T]\Big\},   
\end{align*}
with
\begin{equation}
\left(\partial_t g^i+\xi\partial_x g^i+\frac{g^i}{\epsilon}\right)_i\in \frac{C}{\epsilon}.\label{con:C4}\tag{C4}
\end{equation}
\begin{lemma}\label{lemma:FinTildeC}
If $g\in C$, then $(M[g^1],...,M[g^d])\in C$ and $F(g)\in\tilde C$.
\end{lemma}
\begin{proof}
Let $g\in C$. As in the proof of Proposition \ref{prop:Stability}, we have $M[g^i]\in L^\infty((0,T)_t,L^1((0,\infty)_x\times \mathbb R_\xi))$ and we easily get $(M[g^1],...,M[g^d])\in C$. 
We continue with the proof of $F(g)\in\tilde C$. $F(g)$ is well-defined and Lemma \ref{lemma:SolutionLinearProblem} is applicable (see proof of Proposition \ref{prop:Stability}). Hence, we have $F^i(g)\in C([0,T]_t,L^1((0,\infty)_x\times \mathbb R_\xi))$. Next, we verify (\ref{con:C1} -- \ref{con:C4}) for $F(g)$ and fix $t\in[0,T]$. The characteristics formula for $F(g)$ in (\ref{eq:SolutionLinearProblem}) and $\Psi[s,F(g)(s,0,\cdot)]\in D$ a.e. imply that $F^i(g)_0\ge 0$ a.e. $x,\xi$. If we assume $F^i(g)_0=0$ and use again (\ref{eq:SolutionLinearProblem}), we get $F^i(g)_1=0$ a.e. since $f^{0,i},M[g^i],\Psi^i[s,F(g)(s,0,\cdot))] \in D$ a.e. $s,x,\xi$. This proves (\ref{con:C1}). 
Using Jensen's inequality with the convex function $H$ gives
\begin{align}
\iint_{(0,\infty)\times \mathbb R}&H(F^i(g)(t,x,\xi),\xi)\;\text{d}x\text{d}\xi\nonumber\\
\le& \iint_{(0,\infty)\times\mathbb R}H(f^{0,i}(x-t\xi,\xi),\xi)e^{-t/\epsilon}\,\mathbbm 1_{\{x>t\xi\}}\;\text{d}x\text{d}\xi\nonumber\\
&+\iint_{(0,\infty)\times (0,\infty)}H(\Psi^i[t-x/\xi,F(g)(t-x/\xi,0,\cdot)](\xi),\xi)e^{-x/(\epsilon \xi)}\,\mathbbm 1_{\{x<t\xi\}}\;\text{d}x\text{d}\xi\nonumber\\
&+\frac{1}{\epsilon}\iint_{(0,\infty)\times \mathbb R}\int_0^{\min(t,x/\xi_+)}H(M[g^i](t-s,x-s\xi,\xi),\xi)e^{-s/\epsilon}\;\text{d}s\,\text{d}x\text{d}\xi\nonumber\\
=& \bigg( \iint_{(0,\infty)\times \mathbb R}H(f^{0,i}(x,\xi),\xi)\,\mathbbm 1_{\{x>-t\xi\}}\;\text{d}x\text{d}\xi\nonumber\\
&+\iint_{(0,t)\times (0,\infty)}H(\Psi^i[s,F(g)(s,0,\cdot)](\xi),\xi)e^{s/\epsilon}\;\text{d}\mu(s,\xi)\nonumber\\
&+\frac{1}{\epsilon}\iiint_{(0,t)\times(0,\infty)\times\mathbb R}H(M[g^i](s,x,\xi),\xi)e^{s/\epsilon}\,\mathbbm 1_{\{x>(s-t)\xi\}}\;\text{d}s\text{d}x\text{d}\xi\bigg)e^{-t/\epsilon}.\label{eq:EstimateHwithBC1}
\end{align}
On the other hand, we have
\begin{align}
\iint_{(0,t)\times (-\infty,0)}&H(F^j(g)(s,0,\xi),\xi)e^{s/\epsilon}\;\text{d}\mu(s,\xi)\nonumber\\
\le& \iint_{(0,t)\times (-\infty,0)}H(f^{0,j}(-s\xi,\xi),\xi)\;\text{d}\mu(s,\xi)\nonumber\\
&+\frac{1}{\epsilon}\iint_{(0,t)\times (-\infty,0)}\int_0^s H(M[g^{j}](s-r,-r\xi,\xi),\xi)e^{(s-r)/\epsilon}\;\text{d}r\,\text{d}\mu(s,\xi)\nonumber\\
=&\iint_{(0,\infty)\times(-\infty,0)} H(f^{0,j}(x,\xi),\xi)\mathbbm 1_{\{x<-t\xi\}}\;\text{d}x\text{d}\xi\nonumber\\
&+\frac{1}{\epsilon}\iiint_{(0,t)\times (0,\infty)\times(-\infty,0)}H(M[g^j](s,x,\xi),\xi)e^{s/\epsilon}\mathbbm 1_{\{x<(s-t)\xi\}}\;\text{d}s\text{d}x\text{d}\xi,\label{eq:EstimateHwithBC2}
\end{align}
by Jensen's inequality. These two estimates and the assumption on the energy production at the junction in (\ref{eq:PsiEntropyConservation}) lead to
\begin{align}
\sum_{i=1}^d A^i &\iint_{(0,\infty)\times\mathbb R}H(F^i(g)(t,x,\xi),\xi)\;\text{d}x\text{d}\xi\nonumber \\
&\le \sum_{i=1}^d A^i\Bigg( \iint_{(0,\infty)\times\mathbb R}H(f^{0,i}(x,\xi),\xi)\;\text{d}x\text{d}\xi \nonumber\\
&\quad+\frac{1}{\epsilon}\iiint_{(0,t)\times(0,\infty)\times\mathbb R}H(M[g^i](s,x,\xi),\xi)e^{s/\epsilon}\;\text{d}s\text{d}x\text{d}\xi\Bigg)\,e^{-t/\epsilon} +\int_0^t b_H(s) e^{(s-t)/\epsilon}\;\text{d}s\nonumber\\
&\le C_H(0)+\frac{1}{\epsilon}\int_0^t\int_0^s b_H(r) e^{(s-t)/\epsilon}\;\text{d}r\text{d}s+\int_0^t b_H(s)e^{(s-t)/\epsilon}\;\text{d}s= C_H(t).\label{eq:EstimateHwithBC3}
\end{align}
We used the entropy minimization principle and the definition of $C_H(t)$ for the second inequality and integration by parts for the equality. This proves (\ref{con:C2}) for all $t\in[0,T]$. The proof of (\ref{con:C3}) works the same but we use the bound on the mass production at the junction in (\ref{eq:PsiMassConservation}). Condition (\ref{con:C4}) is satisfied because $(M[g^1],...,M[g^d])\in C$.
\end{proof}
\begin{lemma}\label{lemma:PropertiesCandtildeC}
The sets $C$ and $\tilde C$ are convex and non-empty, $C$ is compact for the weak topology of $L^1((0,T)_t\times (0,\infty)_{\text{loc},x}\times \mathbb R_\xi)^d$ and $\tilde C$ is closed in $C([0,T],L^1((0,\infty)_{\text{loc},x}\times \mathbb R_\xi))^d$.
\end{lemma}
\begin{proof}
$C$ and $\tilde C$ are convex because $H$ is convex. The constant $f^{0}$ belongs to $C$ and $F(f^0)$ belongs to $\tilde C$ by Lemma \ref{lemma:FinTildeC}.
We continue with the compactness of $C$. We prove that $C_\Omega^i=\{g^i|_{x\in\Omega},g\in C\}$ is equi-integrable for a fixed compact set $\Omega\subset(0,\infty)$. Since Proposition \ref{prop:EstimaresHundfglobal}, $C_\Omega^i$ is uniformly bounded in $L^p((0,T)_t\times \Omega_x\times (-R,R))$ with $p>1$ and 
\begin{equation*}
\sup_{\tilde g\in C^i_\Omega}\iiint_{(0,T)\times\Omega\times \mathbb R\backslash [-R,R]} |\tilde g|\;\text{d}t\text{d}x\text{d}\xi\to 0 \quad \text{as } R\to \infty.
\end{equation*}  
Standard arguments imply the equi-integrability. Since Dunford-Pettis' theorem, the equi-integrability is equivalent to the relative compactness of $C^i_\Omega$ in $L^1((0,T)_t\times \Omega_x\times \mathbb R_\xi)$. It remains to prove that $C_{\Omega,i}$ is closed in weak $L^1((0,T)_t\times \Omega_x\times \mathbb R_\xi)$. Since $C_{\Omega,i}$ is convex, it is enough to show that $C_{\Omega,i}$ is closed in strong $L^1((0,T)_t\times \Omega_x\times \mathbb R_\xi)$. Let $(\tilde g_n)_n$ be a sequence in $C_{\Omega,i}$ which converges to $\tilde g\in L^1((0,T)_t\times \Omega_x\times \mathbb R_\xi)$, where $\tilde g_n$ and $\tilde g$ are extended by 0 outside of $\Omega$. We want to show that the extension of $\tilde g$ is in $C$ or equivalently $\tilde g\in C_{\Omega,i}$. After extraction of a subsequence we have $\tilde g_n(t,\cdot)\to\tilde g(t,\cdot)$ in $L^1(\Omega_x\times \mathbb R_\xi)$ and a.e. $x,\xi$, for a.e. $t\in(0,T)$. $(\tilde g_n)_0\ge 0$ implies $(\tilde g)_0\ge 0$. By Lemma \ref{lemma:EstimatesHandf}, Fatou's lemma and Cauchy-Schwarz' inequality, we get for a.e. $t\in(0,T)$ for any measurable set $\mathcal V\subset \Omega_x\times \mathbb R_\xi$ 
\begin{align*}
\iint_{\mathcal V} |(\tilde g_n)_1(t,x,\xi)|\;\text{d}x\text{d}\xi &\le \liminf_{n\to\infty}\iint_{\mathcal V} \sqrt{2 H(\tilde g_n,\xi)(\tilde g_n)_0}\;\text{d}x\text{d}\xi\\
&\le \liminf_{n\to\infty} \left( \frac{2 C_H(t)}{A^i} \iint_{\mathcal V}(\tilde g_n)_0\;\text{d}x\text{d}\xi\right)^{1/2}\\
&= \left( \frac{2 C_H(t)}{A^i} \iint_{\mathcal V} \tilde g_0\;\text{d}x\text{d}\xi \right)^{1/2}
\end{align*} 
Taking $\mathcal V=\{(x,\xi)\in \Omega_x\times \mathbb R_\xi,\,\tilde g_0(t,x,\xi)=0\}$, we obtain $\tilde g_1(t,x,\xi)=0$ a.e. in $\mathcal V$,  a.e. $t$. Thus, $g(t,x,\xi)\in D$ a.e.. Another argument with Fatou's lemma gives
\begin{align*}
A^i\iint_{\Omega\times\mathbb R}H(\tilde g(t,x,\xi),\xi)\;\text{d}x\text{d}\xi&=A^i\iint_{\tilde g_0>0}H(\tilde g(t,x,\xi),\xi)\;\text{d}x\text{d}\xi\\
&\le A^i\liminf_{n\to\infty}\iint_{\tilde g_0>0}H(g_n(t,x,\xi),\xi)\;\text{d}x\text{d}\xi\\
&\le C_H(t),
\end{align*}
but this is (\ref{con:C2}). $\tilde g\in L^{\infty}((0,T)_t,L^1( \Omega_x\times\mathbb R_\xi))$ and (\ref{con:C3}) follow with a similar application of Fatou's lemma. We conclude that $\tilde g\in C_{\Omega,i}$ which proves the weak compactness of $C$ in $L^1((0,T)_t\times(0,\infty)_{\text{loc},x}\times\mathbb R)^d$. 
The proof of the closedness of $\tilde C$ in $C([0,T]_t,L^1((0,\infty)_{\text{loc},x}\times\mathbb R_\xi))^d$ is similar. (\ref{con:C4}) follows from the compactness of $C$.
\end{proof}
\begin{lemma}
$F\colon \tilde C\to\tilde C$ is continuous with respect to $C([0,T]_t,L^1((0,\infty)_{\mathrm{loc},x}\times \mathbb R_\xi))^d$.
\end{lemma}
\begin{proof}
Let $g_n,g\in\tilde C$ with $g_n\to g$ in $C([0,T]_t,L^1((0,\infty)_{\text{loc},x}\times \mathbb R_\xi))^d$. With the notation of Proposition \ref{prop:Stability} we have $\rho_n\to \rho$ and $\rho_n u_n\to \rho u$ in $C([0,T]_t,L^1_{\text{loc},x}(0,\infty))^d$. Proposition \ref{prop:Stability} gives the existence of a subsequence such that $F(g_n)\to F(g)$ in $C([0,T]_t,L^1((0,\infty)_{\text{loc},x}\times \mathbb R_\xi))^d$, but this implies the continuity of $F$.
\end{proof}
\begin{lemma}
$F(\tilde C)$ is relatively compact in $C([0,T]_t,L^1((0,\infty)_{\mathrm{loc},x}\times \mathbb R_\xi))^d$.
\end{lemma}
\begin{proof}
Let $\{F(g_n),\,n\in\mathbb N\}$ be a sequence in $F(\tilde C)$. Since $\tilde C\subset C$ and Lemma \ref{lemma:PropertiesCandtildeC}, there exists $g\in C$ and a subsequence such that $g_n\rightharpoonup g$ in weak $L^1((0,T)_t\times (0,\infty)_{\text{loc},x}\times \mathbb R_\xi)^d$. Then, with the notation of Proposition \ref{prop:Stability}, $\rho_n\rightharpoonup\rho$ and $\rho_n u_n \rightharpoonup \rho u$ in weak $L^1((0,T)_t\times (0,\infty)_{\text{loc},x})^d$ and by (\ref{con:C4}), we have a $h_n\in C$ such that
\begin{equation*}
\partial_t g_n^i+\xi\partial_x g^i_n+\frac{g_n^i}{\epsilon}=\frac{h_n^i}{\epsilon}.
\end{equation*}
The compactness averaging lemma of \cite{GLPS1988} applied to every pipeline and the equi-integrability of $g$ imply that $\int_{\mathbb R}g_n(t,x,\xi)\;\text{d}\xi$ is compact in $L^1_{\text{loc}}((0,T)_t\times (0,\infty)_x)^d=L^1((0,T)_t\times (0,\infty)_{\text{loc},x})^d$. We conclude that $\rho_n\to \rho$ and $\rho_n u_n\to \rho u$ in strong $L^1((0,T)\times (0,\infty)_{\text{loc},x})^d$. Proposition \ref{prop:Stability} gives the existence of a subsequence such that $F(g_n)\to F(g)$ in $C([0,T]_t,L^1((0,\infty)_{\text{loc},x}\times \mathbb R_\xi))^d$.
\end{proof}
\begin{proof}[Proof of Theorem \ref{thm:ExistenceKineticModel}]
We apply the Tychonoff-Schauder fixed point theorem to $F\colon \tilde C\to\tilde C$.\linebreak $C([0,T]_t,L^1((0,\infty)_{\text{loc},x}\times\mathbb R_\xi))^d$ is a locally convex topological vector space. $\tilde C$ is a non-empty, closed, convex subset of $C([0,T]_t,L^1((0,\infty)_{\text{loc},x}\times\mathbb R_\xi))^d$, $F\colon\tilde C\to\tilde C$ is continuous and $F(\tilde C)$ is relatively compact. We obtain the existence of a fixed point $f\in \tilde C$ verifying $F^i(f)=f^i,i=1,...,d$. This gives the existence of a solution to the kinetic model in $[0,T]$ for every $T>0$. Extracting a diagonal subsequence gives a global (in time) solution. (\ref{eq:DifferentialEquationMoments}) follows from integrating (\ref{eq:KineticEquation}) over $\mathbb R_\xi$ since Proposition \ref{prop:EstimaresHundfglobal}. The estimates\linebreak (\ref{eq:SolutionMassConservation} -- \ref{eq:SolutionEntropyConservation}) follow immediately from the fact that $f\in \tilde C$ or more precisely from the conditions (\ref{con:C2}) and (\ref{con:C3}) for the fixed point $f\in\tilde C$.
\end{proof}
\begin{remark}
Notice that the local mass and energy estimates (\ref{eq:PsiMassConservation} -- \ref{eq:PsiEntropyConservation}) are used to prove (\ref{eq:EstimateHwithBC3}) or more precisely to prove the conditions (\ref{con:C2} -- \ref{con:C3}) for $F(g)$. The local estimates are essential to prove the contraction property $F(g)\in\tilde C$ and they give enough compactness to use the fixed point theorem. Furthermore, they imply the global estimates (\ref{eq:SolutionMassConservation} -- \ref{eq:SolutionEntropyConservation}).
\end{remark}
\section{Maximum principle}
\label{sec:MaximumPrinciple}
In this section, we prove kinetic invariance and a maximum principle for a subclass of coupling conditions which are compatible with the so-called kinetic invariant domains. 
\begin{definition}
We call $(\tilde D_\xi^1,\dots,\tilde D_\xi^d)$ a family of kinetic invariant domains for $\Psi$ if
\begin{gather}
\text{ for all }i,\quad f^{0,i}(x,\xi)\in\tilde D^i_\xi, \quad \text{ a.e. } x,\xi,
\end{gather}
implies
\begin{equation}
\text{ for all } t,i,\quad f^i(t,x,\xi)\in \tilde D^i_\xi, \quad \text{ a.e. } x,\xi,
\end{equation}
where $(f^1,...,f^d)$ is the solution obtained in Theorem \ref{thm:ExistenceKineticModel}.
\end{definition}
We start with a characterization of kinetic invariant domains by certain entropies. Notice that the first equivalency was already shown in \cite{BeBo2002Kinetic}.
\begin{lemma}\label{lemma:CharacterizationKineticInvariantDomainsEntropy}
Let $f^i\in D,\,\xi\in \mathbb R$ and $-\infty<\omega_{\min}^i<\omega_{\max}^i<\infty$. The following assertions are equivalent:
\begin{enumerate}[label=(\roman*)]
	\item $f^i\in \tilde D_\xi^i=\{f\in D;\,f=0 \text{ or } \omega^i_{\min}\le\omega_1(f,\xi)\le\omega_2(f,\xi)\le \omega^i_{\max}\},$
	\item $H_{S_M^i}(f^i,\xi)\le 0$ and $H_{S_m^i}(f^i,\xi)\le 0$,\\
	where $S_M^i(v)=(v-\omega_{\max}^i)_+^2$ and $S_m^i(v)=(\omega_{\min}^i-v)_+^2$,
	\item $H_{S_\omega^i}(f^i,\xi)\le 0$, where $S_\omega^i(v)=S_M^i(v)+S_m^i(v)$.
\end{enumerate}
Furthermore, $S_M^i,\,S_m^i,\,S_\omega^i$ are positive, convex and of class $C^1$.
\end{lemma}
\begin{proof}
One easily checks that $S_M^i,\,S_m^i,\,S_\omega^i$ are positive, convex, of class $C^1$ and that the corresponding kinetic entropies are positive.
For $\xi\in\mathbb R$, we have 
\begin{align*}
H_{S_M^i}(f^i,\xi)\le 0 &\iff f^i=0 \text{ or }\Phi(\rho(f^i,\xi),u(f^i,\xi),\xi,v)(v-\omega^i_{\max})_+^2=0\text{ a.e. }v\\
&\iff f^i=0 \text{ or }(v-\omega^i_{\max})_+=0 \text{ a.e. in }(\omega_1(f^i,\xi),\omega_2(f^i,\xi))\\
&\iff f^i=0 \text{ or }\omega_2(f^i,\xi)\le \omega^i_{\max}.
\end{align*}
A similar result holds for $S_m^i$ and we get
\begin{equation*}
H_{S_M^i}(f^i,\xi)\le 0 \text{ and }H_{S_m^i}(f^i,\xi)\le 0 \iff f^i\in \tilde D^i_{\xi}.
\end{equation*}
The second equivalence relation follows from the fact, that $H_{S_M^i},\,H_{S_m^i},\,H_{S_\omega^i}\ge 0$.
\end{proof}
\begin{theorem}\label{thm:MaximumPrinciple}
Assume that all assumptions in Theorem \ref{thm:ExistenceKineticModel} hold true and
\begin{equation}
\tilde D^i_\xi=\{f\in D;\,f=0 \text{ or } \omega^i_{\min}\le\omega_1(f,\xi)\le\omega_2(f,\xi)\le \omega^i_{\max}\},
\end{equation}
for $-\infty<\omega_{\min}^i<\omega_{\max}^i<\infty$.
Let $f^{0,i}(x,\xi)\in \tilde D^i_\xi$ a.e. $x,\xi$ and for a.e. $t\in(0,\infty)$
\begin{equation}\label{eq:DefPsiKineticInvariant}
\sum_{i=1}^d A^i\int_0^\infty |\xi|\ H_{S_\omega^i}(\Psi^i[t,g](\xi),\xi)\;\mathrm{d}\xi\,
\le\sum_{i=1}^d A^i \int_{-\infty}^0|\xi|\ H_{S_\omega^i}(g^i(\xi),\xi)\;\mathrm{d}\xi,
\end{equation}
for all $g\in L^1_\mu((-\infty,0)_\xi,D)^d$. $(\tilde D^1_\xi,\dots,\tilde D^d_\xi)$ is a family of convex kinetic invariant domains for $\Psi$.
The sets $\tilde D^i_\xi$ are associated with the invariant domains 
\begin{equation}
\tilde D^i=\{(\rho,u)\in [0,\infty)\times\mathbb R;\,\rho=0 \text{ or }\omega^i_{\min}\le\omega_1(\rho,u)\le\omega_2(\rho,u)\le\omega^i_{\max}\}
\end{equation}
of the isentropic gas equations (\ref{eq:MacroscopicEquation}) in the following sense:
\begin{enumerate}[label=(\roman*)]
  \item If $(\rho,u)\in \tilde D^i$, then $M(\rho,u,\xi)\in\tilde D^i_\xi$ for all $\xi\in\mathbb R$.
  \item For any $f\in L^1(\mathbb R_\xi)$ such that $f(\xi)\in \tilde D^i_\xi$ a.e. $\xi$, the averages $(\rho,\rho u)=\int_{\mathbb R}f(\xi)\;\text{d}\xi$ verify $(\rho,u)\in\tilde D$.
\end{enumerate} 
Furthermore, if $\xi\not\in [\omega^i_{\min},\omega^i_{\max}]$, then $\tilde D^i_\xi =\{0\}$.
\end{theorem}
\begin{proof}
Since Proposition \ref{prop:ConvexityofSandH} and Lemma \ref{lemma:CharacterizationKineticInvariantDomainsEntropy}, $\tilde D^i_\xi$ and $\tilde D^1_\xi\times ... \times \tilde D^d_\xi$ are convex. Let $f$ be the solution obtained in Theorem \ref{thm:ExistenceKineticModel}.
For the kinetic invariance, we proceed as in (\ref{eq:EstimateHwithBC1} -- \ref{eq:EstimateHwithBC3}) and get 
\begin{align*}
&\!\!\iint_{(0,\infty)\times\mathbb R}H_{S_\omega^i}( f^i(t,x,\xi),\xi)\;\text{d}x\text{d}\xi\\
&\le\iint_{(0,\infty)\times\mathbb R}H_{S_\omega^i}( f^i(0,x,\xi),\xi)\;\text{d}x\text{d}\xi+\iint_{(0,t)\times (0,\infty)}H_{S_\omega^i}( f^i(t,0,\xi),\xi)\;\text{d}t\text{d}\xi=0
\end{align*}
for all $t\in[0,\infty)$. With Lemma \ref{lemma:CharacterizationKineticInvariantDomainsEntropy}, we conclude that $(\tilde D^1_\xi,...,\tilde D^d_\xi)$ is a family of kinetic invariant domains for $\Psi$. The relation between $\tilde D_\xi^i$ and $\tilde D^i$ was proven in \cite[Theorem 1.4]{BeBo2002Kinetic}.
\end{proof}
\begin{remark}
Since we introduced the additional assumption (\ref{eq:DefPsiKineticInvariant}), the coupled half-space solutions depend only on $\Psi[t,g]$ with $g\in\bigtimes_{i=1}^d L_\mu^1((-\infty,0)_\xi,\tilde D_\xi^i)$. Therefore, it is equivalent to define a coupling function
\begin{equation*}
\tilde \Psi\colon (0,\infty)_t\times \bigtimes_{i=1}^d L^1_\mu((-\infty,0)_\xi,D_\xi^i)\to\bigtimes_{i=1}^d L^1_\mu((0,\infty)_\xi,D_\xi^i)
\end{equation*}
and to extend it by zero for $g\not\in \bigtimes_{i=1}^d L^1_\mu((-\infty,0)_\xi,D_\xi^i)$.
\end{remark}
\begin{proposition}\label{prop:UniformLinftyBoundsinEpsilon}
Let $f^{0}$ and $\Psi$ be as in Theorem \ref{thm:MaximumPrinciple}. Then $\rho_\epsilon^i,\, u^i_\epsilon,\,f^i_\epsilon,\,M[f^i_\epsilon]$ are uniformly bounded in $L^\infty$. Furthermore, we have $\operatorname{supp}_\xi f^i_\epsilon\subset[\omega^i_{\min},\omega^i_{\max}]$, $\operatorname{supp}_\xi M[f^i_\epsilon]\subset[\omega^i_{\min},\omega_{\max}^i]$ and $|(f_\epsilon^i)_1|\le A (f_\epsilon^i)_0$ for a constant $A>0$.
\end{proposition}
\begin{proof}
This follows from Theorem \ref{thm:MaximumPrinciple} and Lemma \ref{lemma:Bijectionfandomega}.
\end{proof}
\begin{corollary}\label{cor:UniformIntegralBoundinEpsilon}
Let $S\colon\mathbb R\to\mathbb R$ be convex, of class $C^1$ and satisfies $|S(v)|\le B\,(1+v^2)$ for a constant $B\ge 0$.
\begin{enumerate}[label=(\roman*)]
	\item The sequence $(t,x,\xi)\mapsto H_S(f^i_\epsilon(t,x,\xi),\xi)$ is bounded in $C([0,\infty)_t,L^1((0,\infty)_x\times \mathbb R_\xi))$. 
	\item The sequence $(t,x,\xi)\mapsto H_S(f^i_\epsilon(t,x,\xi),\xi)$ is bounded in $C([0,\infty)_x,L^1_\mu((0,T)_t\times \mathbb R_\xi))$.
\end{enumerate}
\end{corollary}
\begin{proof}
The boundedness in \textit{(i)} follows from Proposition \ref{prop:UniformLinftyBoundsinEpsilon}, the definition of $H$ and the upper bound on $S$. 
Lebesque's theorem, the continuity of $H_S$ in $\{f\in D,\,|f_1|\le A f_0\}$ (Proposition \ref{prop:SmoothnessPropertiesHandS}) and \linebreak $f^i_\epsilon\in C([0,\infty)_t,L^1((0,\infty)_x\times \mathbb R_\xi))$ give the continuity. 
Part \textit{(ii)} works similar but we use\linebreak $f^i_\epsilon\in C([0,\infty)_x,L_\mu^1((0,T)_t\times\mathbb R_\xi))$.
\end{proof}
We end the section with a relation for the kinetic entropy fluxes at the junction.
\begin{proposition}\label{prop:ExistencebSforLinftysolutions}
Let $f^0$ and $\Psi$ be as in Theorem \ref{thm:MaximumPrinciple}.
Let $\mathcal S=(S^1,\dots,S^d)$, with convex functions $S^i\colon\mathbb R\to\mathbb R$ of class $C^1$ and $|S^i(v)|\le B^i\,(1+v^2)$ for constants $B^i\ge 0$.
Then, there exists a function $b_{\mathcal S}\in L^\infty_t(0,\infty)\subset L^1_{\mathrm{loc},t}(0,\infty)$ such that for a.e. $t>0$
\begin{equation*}
\sum_{i=1}^d A^i \int_0^\infty |\xi|\,H_{S^i}(\Psi^i[t,g](\xi),\xi)\;\text{d}\xi\le  \sum_{i=1}^d A^i \int_0^\infty |\xi|\,H_{S^i}(g,\xi)\;\text{d}\xi+b_{\mathcal S}(t),
\end{equation*}
for all $g\in\bigtimes_{i=1}^d L^1_\mu((-\infty,0)_\xi,\tilde D^i_\xi)$. Furthermore, we have for a.e. $t\in(0,\infty)$
\begin{equation*}
\sum_{i=1}^d A^i \int_{\mathbb R} \xi\,H_{S^i}(\Psi^i[t,f(t,0,\cdot)](\xi),\xi)\;\text{d}\xi\le b_{\mathcal S}(t),
\end{equation*}
for all solutions $f$ obtained in Theorem \ref{thm:ExistenceKineticModel}.
\end{proposition}
\begin{proof}
Since (\ref{eq:DefPsiKineticInvariant}), we have $\Psi^i[t,g]\in \tilde D_\xi^i$ for a.e. $t,\xi$ and get a $L^\infty$-bound for $\Psi^i[t,g]$ independent of $g$. The first part follows from the definition of $H$ and $|S^i(v)|\le B^i(1+v^2)$.
As in (\ref{eq:EstimateHwithBC2}) and with Lemma \ref{lemma:CharacterizationKineticInvariantDomainsEntropy}, we get $f(t,0,\xi)\in \tilde D_\xi^i$ a.e. $t,\xi$. The claim follows from the first part.
\end{proof}
\section{Relaxation to the Macroscopic Limit}
\label{sec:RelaxationtotheMacroscopicLimit}
In this section we prove convergence of the kinetic solutions for $\epsilon\to 0$ based on the arguments in \cite{BeBo2002Boundary}.
Until the end of this section, we assume that the assumptions in Theorem \ref{thm:ExistenceMacroscopicEquation} are satisfied. 
\subsection{Interior Relaxation}
Since part (v) in Proposition \ref{prop:SmoothnessPropertiesHandS}, Proposition \ref{prop:SubdifferentialInequality} and Proposition \ref{prop:UniformLinftyBoundsinEpsilon}, we have $H'_S(f^i_\epsilon,\xi)\in L^\infty((0,T)_t\times (0,\infty)_x\times \mathbb R_\xi)$. A modification of Theorem 1.1 in \cite{Bo2001} for vector-valued equations gives
\begin{equation}
\partial_t (H_S(f_\epsilon^i,\xi))+\xi\partial_x(H_S(f_\epsilon^i,\xi))=H'_S(f^i_\epsilon,\xi)\frac{M[f_\epsilon^i]-f_\epsilon^i}{\epsilon},
\end{equation}
and $M[f^i_\epsilon]-f^i_\epsilon=0$ a.e. where $f^i_\epsilon=0$.
Let $\varphi^i\in\mathcal D([0,\infty)_t\times[0,\infty)_x)$. Using the continuity properties in Corollary \ref{cor:UniformIntegralBoundinEpsilon} justifies
\begingroup
\allowdisplaybreaks
\begin{align}
\begin{split}
&-\iiint_{(0,\infty)^2\times \mathbb R}H_S(f_\epsilon^i,\xi)\,\partial_t\varphi^i\;\text{d}t\text{d}x\text{d}\xi -\iint_{(0,\infty)\times \mathbb R}H_S(f_\epsilon^i(t=0),\xi)\,\varphi^i(0,x)\;\text{d}x\text{d}\xi\\
&-\iiint_{(0,\infty)^2\times \mathbb R}\xi\,H_S(f_\epsilon^i,\xi)\,\partial_x\varphi^i\;\text{d}t\text{d}x\text{d}\xi -\iint_{(0,\infty)\times \mathbb R}\xi\,H_S(f_\epsilon^i(x=0),\xi)\,\varphi^i(t,0)\;\text{d}t\text{d}\xi\\
&=\iiint_{(0,\infty)^2\times \mathbb R}H'_S(f_\epsilon^i,\xi)\frac{M[f^i_\epsilon]-f^i_\epsilon}{\epsilon}\varphi^i\;\text{d}t\text{d}x\text{d}\xi\\
&=\iiint_{(0,\infty)^2\times \mathbb R}\left(H'_S(f_\epsilon^i,\xi)-T_S(\rho^i_\epsilon,u^i_\epsilon)\right)\frac{M[f^i_\epsilon]-f^i_\epsilon}{\epsilon}\varphi^i\;\text{d}t\text{d}x\text{d}\xi ,\label{eq:IntegralKineticDissipation}
\end{split}
\end{align}
\endgroup
where we used that $T_S(\rho_\epsilon,u_\epsilon)$ is independent of $\xi$ and the definition of the Maxwellian for the second equality.
By Proposition \ref{prop:SubdifferentialInequality}, we obtain that
\begin{gather}
-\iiint_{(0,\infty)^2\times \mathbb R}H_S(f_\epsilon^i,\xi)\,\partial_t\varphi^i\;\text{d}t\text{d}x\text{d}\xi-\iiint_{(0,\infty)^2\times \mathbb R}\xi\,H_S(f_\epsilon^i,\xi)\,\partial_x \varphi^i\;\text{d}t\text{d}x\text{d}\xi\nonumber\\
-\iint_{(0,\infty)\times \mathbb R}\xi\,H_S(f_\epsilon^i(x=0),\xi)\,\varphi^i(t,0)\;\text{d}t\text{d}\xi\nonumber\le 0,\nonumber
\end{gather}
for $\varphi^i\in\mathcal D((0,\infty)_t\times [0,\infty)_x),\,\varphi^i\ge 0$ or equivalently
\begin{gather}
\begin{split}
-\iint_{(0,\infty)^2}\eta_S(\rho_\epsilon^i,u_\epsilon^i)\,\partial_t\varphi^i\;\text{d}t\text{d}x-\iint_{(0,\infty)^2}G_S(\rho_\epsilon^i,u_\epsilon^i)\,\partial_x\varphi^i\;\text{d}t\text{d}x\label{eq:KineticIntegralEntorpyEquation}\\
-\iint_{(0,\infty)\times \mathbb R}\xi\,H_S(f_\epsilon^i(x=0),\xi)\,\varphi^i(t,0)\;\text{d}t\text{d}\xi\,-\,\langle R_{S,\epsilon}^i,\varphi^i\rangle \,\le 0,
\end{split}
\end{gather}
with 
\begin{align}
\begin{split}
\langle R_{S,\epsilon}^i,\varphi^i\rangle=\iiint_{(0,\infty)^2\times \mathbb R}(H_S(f_\epsilon^i,\xi)-H_S(M[f_\epsilon^i],\xi))\,\partial_t \varphi^i\;\text{d}t\text{d}x\text{d}\xi\\
+\iiint_{(0,\infty)^2\times \mathbb R}\xi (H_S(f_\epsilon^i,\xi)-H_S(M[f_\epsilon^i],\xi))\,\partial_x\varphi^i\;\text{d}t\text{d}x\text{d}\xi.
\end{split}
\end{align}
Since (\ref{eq:IntegralKineticDissipation}) and Corollary \ref{cor:UniformIntegralBoundinEpsilon},
\begin{align}
\begin{split}
\iint_{(0,\infty)\times \mathbb R}& H_S(f_\epsilon^i(t=T),\xi)\;\text{d}x\text{d}\xi -\iint_{(0,\infty)\times \mathbb R}H_S(f_\epsilon^i(t=0),\xi)\;\text{d}x\text{d}\xi\\
&\quad -\iiint_{(0,\infty)\times \mathbb R}\xi\,H_S(f_\epsilon^i(x=0),\xi)\;\text{d}t\text{d}\xi\\
&=\iiint_{(0,T)\times(0,\infty)\times \mathbb R}\left(H'_S(f_\epsilon^i,\xi)-T_S(\rho^i_\epsilon,u^i_\epsilon)\right)\frac{M[f^i_\epsilon]-f^i_\epsilon}{\epsilon}\;\text{d}t\text{d}x\text{d}\xi.
\end{split}
\end{align}
Proposition \ref{prop:UniformLinftyBoundsinEpsilon} implies that
\begin{align}
\begin{split}\label{eq:boundedHprimeL1norm}
Q_{S,\epsilon}^i=\int_{\mathbb R}(H'(f_\epsilon^i,\xi)-T_{v^2/2}(\rho_\epsilon^i,u_\epsilon^i))\frac{M[f_\epsilon^i]-f_\epsilon^i}{\epsilon}\;\text{d}\xi,\quad \epsilon>0,\\
\text{is uniformly bounded in $L^1((0,T)_{t}\times (0,\infty)_x)$ for every $T>0$.}
\end{split}
\end{align}
This, together with the fact that $f_\epsilon^i,M[f_\epsilon^i]$ are bounded in $L^\infty((0,\infty)_t\times(0,\infty)_x\times \mathbb R_\xi)$ and the property of uniform compact support implies $f_\epsilon^i-M[f_\epsilon^i]\to 0$ a.e. $t,x,\xi$ with the arguments of Proposition 6.2 in \cite{BeBo2002Kinetic}. \\
Next, we prove the convergence $R^i_{S,\epsilon}\to 0$ in $W^{-1,p}_{\mathrm{loc}}$. We have that
\begin{equation}\label{eq:ConvergenceRemainderR1}
0\le \int_{\mathbb R}H_S(f^i_\epsilon,\xi)-H_S(M[f_\epsilon^i],\xi)\;\text{d}\xi\le\int_{\mathbb R}H'_S(f_\epsilon^i,\xi)\cdot (f_\epsilon^i-M[f_\epsilon^i])\;\text{d}\xi\to 0
\end{equation}
in $L^1_{\mathrm{loc}}((0,\infty)_t\times (0,\infty)_x)$, since (\ref{eq:boundedHprimeL1norm}). The same holds true for 
\begin{equation}\label{eq:ConvergenceRemainderR2}
\int_{\mathbb R}\xi\,(H_S(f^i_\epsilon,\xi)-H_S(M[f_\epsilon^i],\xi))\;\text{d}\xi,
\end{equation}
because $f_\epsilon^i-M[f_\epsilon^i]\to 0$ a.e. and the fact that $f_\epsilon$ has uniform compact support w.r.t. $\xi$ (see Proposition 6.4 in \cite{BeBo2002Kinetic}).
Since we have also boundedness of (\ref{eq:ConvergenceRemainderR1} -- \ref{eq:ConvergenceRemainderR2}) in $L^\infty((0,T)_t\times(0,R)_x)$, we get convergence in $L^p_{\mathrm{loc}}((0,\infty)_t\times(0,\infty)_x)$ for any $1\le p<\infty$.
We conclude that $ R_{S,\epsilon}^i\to 0$ in $W^{-1,p}_{\mathrm{loc}}$ for any $1<p<\infty$. Then, (\ref{eq:IntegralKineticDissipation}) with $\varphi^i\in\mathcal D((0,\infty)_t\times(0,\infty)_x)$ reads
\begin{equation}
\partial_t\eta_S(\rho_\epsilon^i,u_\epsilon^i)+\partial_x G_S(\rho_\epsilon^i,u_\epsilon^i)= Q_{S,\epsilon}^i+R_{S,\epsilon}^i,
\end{equation}
where 
\begin{gather}
\begin{split}
Q_{S,\epsilon}^i \quad \text{ lies in a bounded set of the space of measures and}\\
R_{S,\epsilon}^i\to 0 \quad \text{in }W^{-1,p}_{loc}\text{ for any $1<p<\infty$ as $\epsilon\to 0$}.
\end{split}
\end{gather}
Since $\rho_\epsilon^i,u_\epsilon^i$ are bounded in $L^\infty$, we can apply the compensated compactness result of \cite{LPS1996}. We summarize that, up to a subsequence, $(\rho_\epsilon^i,\rho_\epsilon^i u_\epsilon^i)$ converge a.e. in $(0,\infty)\times \mathbb R$ to an entropy solution $(\rho^i,\rho^i u^i)$ of (\ref{eq:MacroscopicEquation}), (\ref{eq:EntropySolution}). Furthermore, we have $(\rho^i,\rho^i u^i)\in\tilde D^i$ a.e. $x,t$ and the initial data is attained in the sense $\overline{(\rho^i,\rho^i u^i)}(x=0)=\int_{\mathbb R}f^{0,i}\;\text{d}\xi$ of the weak trace. The weak entropy flux boundary traces $\overline{G_S(\rho^i,u^i)}(t,0)$ exist and are unique since Theorem \ref{thm:WeakTraceDivMeasure}.
\subsection{Boundary Relaxation}
Next, we consider the relaxation at the boundary. For $S\colon\mathbb R\to\mathbb R$ convex, of class $C^1$ with $|S(v)|\le B(1+v^2)$ and $\epsilon>0$, we define
\begin{equation}
\psi^i_{S,\epsilon}(t)=\int_\mathbb R \xi\, H_S(f^i_\epsilon(t,x,\xi),\xi)\;\text{d}\xi,\quad t>0.
\end{equation}
The sequence $(\psi^i_{S,\epsilon})_{\epsilon>0}$ is bounded in $L^\infty_t(0,\infty)$ and there exists $\psi_S^i\in L^\infty_t(0,\infty)$ such that 
\begin{equation}
\psi_{S,\epsilon}^i\rightharpoonup \psi^i_S \quad \text{in }L^\infty_{w*}(0,\infty)\quad \text{as }\epsilon\to 0,
\end{equation}
after passing if necessary to a subsequence. Next, we derive a relation between $\psi^i_S$ and the weak traces $\overline{G_S(\rho^i, u^i)}$.
\begin{proposition}\label{prop:InequalityTraces}
Let all assumptions of Theorem \ref{thm:ExistenceMacroscopicEquation} be satisfied and fix $S\colon\mathbb R\to\mathbb R$ convex, of class $C^1$, with $|S(v)|\le B(1+v^2)$, then
\begin{equation*}
\overline{G_S(\rho^i,u^i)}(t,0)\le \psi^i_S(t)\quad \text{a.e. }t>0.
\end{equation*}
Furthermore, we have equality if $S(v)\in\{\mathbbm 1,v\}$.
\end{proposition}
\begin{proof}
We recall from (\ref{eq:KineticIntegralEntorpyEquation}), that 
\begin{gather*}
-\iint_{(0,\infty)^2}\eta_S(\rho_\epsilon^i,u_\epsilon^i)\,\partial_t\varphi^i\;\text{d}t\text{d}x-\iint_{(0,\infty)^2}G_S(\rho_\epsilon^i,u_\epsilon^i)\,\partial_x\varphi^i\;\text{d}t\text{d}x\\
-\int_{(0,\infty)}\psi^i_{S,\epsilon}(t)\,\varphi^i(t,0)\;\text{d}t\,-\,\langle R_{S,\epsilon}^i,\varphi^i\rangle \,\le 0,
\end{gather*}
for $\varphi^i\in\mathcal D((0,\infty)_t\times[0,\infty)_x),\,\varphi^i\ge 0$.
Taking the limit gives
\begin{gather*}
-\iint_{(0,\infty)^2}\eta_S(\rho^i,u^i)\,\partial_t\varphi^i\;\text{d}t\text{d}x-\iint_{(0,\infty)^2}G_S(\rho^i,u^i)\,\partial_x\varphi^i\;\text{d}t\text{d}x\\
-\int_{(0,\infty)}\psi^i_S(t)\,\varphi^i(t,0)\;\text{d}t\,\le 0
\end{gather*}
for a subsequence $\epsilon\to 0$.
Using Theorem \ref{thm:WeakTraceDivMeasure} with $(\eta_S,G_S)$ leads to
\begin{gather*}
\iint_{(0,\infty)^2}\operatorname{div}_{t,x}(\eta_S(\rho^i,u^i),G_S(\rho^i,u^i))\,\varphi^i\;\text{d}t\text{d}x\\
+\int_0^\infty \left(\overline{G_S(\rho^i,u^i)}(t,0)- \psi_S^i(t)\right)\,\varphi^i(t,0)\;\text{d}t\,\le 0.
\end{gather*}
We set $\varphi^i(t,x)=\varphi^i_{1,h}(x)\,\varphi^i_{2}(t)$ with $\varphi^i_{1,h}(x)=1$ for $x\le h/2$, $\varphi^i_{1,h}(x)=0$ for $x\ge h$ and $|(\varphi^i_{1,h})'|\le C/ h$. We take the limit $h\to0$ with Lebesque's theorem for the measure $\operatorname{div}_{t,x}(\eta_S(\rho^i,u^i),G_S(\rho^i,u^i))$ and get
\begin{gather*}
\int_0^\infty \left(\overline{G_S(\rho^i,u^i)}(t,0)- \psi^i_S(t)\right)\,\varphi_2^i(t)\;\text{d}t\text{d}\xi\,\le 0,
\end{gather*}
for every $\varphi^i_2\in\mathcal D((0,\infty)_t)$.
\end{proof}
This completes the proof of Theorem \ref{thm:ExistenceMacroscopicEquation}.
\begin{proof}[Proof of Corollary \ref{cor:MonotonicityGammaLimit}]
The uniform bound on $\Gamma[t,\cdot]$ ensures that the obtained quantities are still bounded functions. 
The result follows from Proposition \ref{prop:InequalityTraces} and the monotonicity of $\Gamma[t,\cdot]$.
\end{proof}
\section{Examples}
\label{sec:Examples}
In this section, we give examples for coupling functions which fit in the framework of Theorem \ref{thm:ExistenceMacroscopicEquation}. In the first part, we define three general classes of coupling functions and derive some of their basic properties. In the second part, we give more explicit coupling and boundary conditions and show that they fit in our framework. We begin with a remark about the physical interpretation of the functions $b_{\Gamma,\mathcal S}$.
\begin{remark}Let all assumptions in Theorem \ref{thm:ExistenceMacroscopicEquation} be satisfied. Let $i_l,S_l, \,l=1,\dots,k$ and $\Gamma\colon (0,\infty)_t\times \mathbb R^k\to \mathbb R$ be as usual with
\begin{equation*}
\Gamma[t,\overline{G_{S_1}(\rho^{i_1},u^{i_1})}(t,0),\dots,\overline{G_{S_k}(\rho^{i_k},u^{i_k})}(t,0)]\le b_{\Gamma, \mathcal S}(t) \quad \text{a.e. }t>0,
\end{equation*}
where $b_{\Gamma, \mathcal S}\in L^\infty_t(0,\infty)$ is independent of the initial data.
It is important to observe that $b_{\Gamma, \mathcal S}$ depends strongly on the choice of the kinetic invariant domains $\tilde D^i_\xi$:\\
Set for example $d=2$ and $\Psi^i[t,g](\xi)=(g^j_0(-\xi),-g^j_1(-\xi)),\,t>0,\,\xi>0,\, i\neq j$. The best function $b_{\Gamma,\mathcal S}(t)$ for $\Gamma[t,G_{S^1},G_{S^2}]=\sum_{i=1}^2 A^i G_{S^i},\mathcal S=(\mathbbm 1,0)$ is given by $b_{\Gamma,\mathcal S}(t)=\sup\{\int_{(-\infty,0)} |\xi| g_0(\xi)\;\mathrm{d}\xi;\, g\in L^1((-\infty,0)_\xi,D_\xi^2)\}$, but this constant depends on $D_\xi^2$ and goes to infinity as $\omega_{\min}^2\to -\infty$.
For several examples in this section, we get functions $b_{\Gamma,\mathcal S}$ which are independent of the kinetic invariant domain and depend only on $\Psi$. Such a behavior was expectable, since the $L^\infty$-bounds on the initial data and the kinetic invariant domains $D_\xi^i$ were introduced for technical reasons and are unphysical.
\end{remark}
\subsection{Maxwellian Coupling Conditions}
\label{sec:MaxwellianCouplingConditions}
Since $\Psi$ is used to couple the half-space problems on the kinetic level, we expect that some information will be lost, if we take the limit. Therefore, we are especially interested in the behavior of half-moments of $f$ and $H_S(f,\xi)$. As in \cite{BeBo2002Boundary} it can be useful to define the outgoing data to be the Maxwellian of certain macroscopic variables $\hat\rho^{i},\,\hat u^{i}$ depending on the incoming data. For a given coupling condition $\Psi$, we construct a coupling function $\hat{\Psi}$ with Maxwellian outgoing data by
\begin{gather}
\begin{split}\label{eq:DefHatPsi}
&\hat{\Psi}^i[t,g](\xi)=M(\hat\rho^{i},\hat u^{i},\xi),\quad \xi>0, \text{ where }(\hat\rho^i,\hat u^{i}) \text{ satisfy} \\
&\int_0^\infty \xi \,M(\hat \rho^{i},\hat u^{i},\xi)\;\text{d}\xi=\int_0^\infty \xi \,\Psi^i[t,g]\;\text{d}\xi.
\end{split}
\end{gather}
We get the following result:
\begin{proposition}
Let $\Psi$ be defined as in (\ref{eq:DefPsi}). Then, $\hat\Psi\colon (0,\infty)\times L^1_\mu((-\infty,0)_{\xi},D)^d \to  L^1_\mu((0,\infty)_{\xi},D)^d$ as in (\ref{eq:DefHatPsi}) is well-defined.
Furthermore, if the assumptions in Theorem \ref{thm:ExistenceMacroscopicEquation} are satisfied for $\Psi$, then the same holds true for $\hat{\Psi}$.
\end{proposition}
\begin{proof}
To prove that $\hat\Psi$ is well-defined, we have to show that $\{f\in \mathbb R^2;\,f_0>0\}$ is in bijection with $\{(\rho,u)\in (0,\infty)\times\mathbb R);\,\omega_2(\rho,u)>0\}$ by $\int_0^\infty \xi \,M(\hat \rho,\hat u,\xi)\;\text{d}\xi=f$. One can prove this with monotonicity properties with respect to the Riemann invariants $\omega_1,\,\omega_2$. The entropy flux inequalities for $\hat\Psi$ follow from Proposition \ref{prop:SubdifferentialInequality}. Therefore, it remains to prove (\ref{eq:ContPsi}).\\
Since the obtained solution will only depend on $\Psi[t,g]$ with $g^i\in L_\mu^1((-\infty,0)_\xi,\tilde D_\xi^i)$, we can set $\hat\Psi$ to zero for $g^i\not \in L_\mu((-\infty,0)_\xi,\tilde D_\xi^i)$. We take $g^i_n\in L_\mu^1((0,\infty)_{\mathrm{loc},t}\times(-\infty,0)_\xi,\tilde D_\xi^i)$ converging to $g^i$ in $L_\mu^1((0,\infty)_{\mathrm{loc},t}\times(-\infty,0)_\xi,D)$. Proposition \ref{prop:SubdifferentialInequality} and (\ref{eq:KineticInvariancePsi}) imply that $\hat\Psi[t,g_n(t,\cdot)](\xi)$ is uniformly bounded in $L^\infty_\mu((0,\infty)_t\times(0,\infty)_\xi,D)^d$. Since Lebesque's theorem, it remains to prove point-wise convergence a.e. $t,\xi$. Since (\ref{eq:ContPsi}), we can take a subsequence such that $\Psi[t,g_n(t,\cdot)](\xi)\to \Psi[t,g(t,\cdot)](\xi)$ in $L^1_\mu((0,\infty)_\xi,D)^d$ for a.e. $t>0$. Since $f\mapsto (\hat\rho,\hat u)$ with $\int_0^\infty \xi \,M(\hat \rho,\hat u,\xi)\;\text{d}\xi=f$ is continuous on $\{f\in D;\,|f_1|\le A f_0\}$, we get $(\hat \rho_n,\hat u_n)(t)\to(\hat \rho,\hat u)(t)$ for a.e. $t>0$. This implies $\hat\Psi[t,g_n(t,\cdot)](\xi)\to \hat\Psi[t,g_n(t,\cdot)](\xi)$ a.e. $t,\xi$ and we get the result.
\end{proof}
\subsection{Linear Coupling Conditions}
Next, we introduce a simple class of linear coupling functions for $d\in \mathbb N$ pipelines. Let 
\begin{equation}\label{eq:defcij}
c^{ij}\ge 0 \text{ be such that } \sum_{j=1}^d c^{ij}=1 \text{ and }\sum_{i=1}^d A^i c^{ij}=A^j.
\end{equation}
Notice that the second condition is satisfied after possibly taking new $\tilde A^i$.
We define the coupling function by
\begin{equation}\label{eq:DefCouplingConditionLinear}
\Psi^{c,i}[t,g](\xi)=\sum_{j=1}^d  c^{ij} \begin{pmatrix}
g_0^j(-\xi)\\-g_1^j(-\xi)\end{pmatrix},\quad \xi>0.
\end{equation}
Furthermore, we fix a tuple $\mathcal S=(S^1,\dots,S^d)\in C^1(\mathbb R,\mathbb R^d)$ of convex functions with $|S^i(v)|\le B^i(1+v^2)$ and
\begin{equation}\label{eq:CouplingNetworkConditioncijandS}
A^{j} S^j(v)=\sum_{i=1}^d A^{i} c^{ij} S^i(-v),\quad \text{for every }v\in\mathbb R.
\end{equation}
Since (\ref{eq:defcij}), this condition is satisfied for $S^i(v)=S(v)$ and $S(v)=S(-v)$. Since Proposition \ref{prop:ConvexityofSandH} and the definition of $H_{S^i}$, we get
\begin{align*}
\sum_{i=1}^d A^i\int_0^\infty |\xi|\,H_{S^i}(\Psi^{c,i}[t,g](\xi),\xi)\;\text{d}\xi
&=\sum_{i=1}^d A^i\int_0^\infty |\xi|\,H_{S^i}(\sum_{j=1}^d c^{ij}\begin{pmatrix}g_0^j(-\xi)\\-g_1^j(-\xi)\end{pmatrix},\xi)\;\text{d}\xi\\
&\le\sum_{i=1}^d \sum_{j=1}^d A^{i} c^{ij}\int_0^\infty |\xi|\,H_{S^i}(\begin{pmatrix}g_0^j(-\xi)\\-g_1^j(-\xi)\end{pmatrix},\xi)\;\text{d}\xi\\
&=\sum_{i=1}^d  \sum_{j=1}^d A^{i} c^{ij}\int_{-\infty}^0 |\xi|\,H_{S^i}(\begin{pmatrix}g_0^j(\xi)\\-g_1^j(\xi)\end{pmatrix},-\xi)\;\text{d}\xi\\
&=\sum_{i=1}^d \sum_{j=1}^d  A^{i} c^{ij}\int_{-\infty}^0 |\xi|\,H_{S^i(-\cdot)}(g^j(\xi),\xi)\;\text{d}\xi\\
&=\sum_{j=1}^d A^j\int_{-\infty}^0 |\xi|\,H_{S^j}(g^{j}(\xi),\xi)\;\text{d}\xi,
\end{align*}
or equivalently
\begin{equation}\label{eq:EntropyFluxInequalityNetworkKinetic}
\sum_{i=1}^d A^i \int_{\mathbb R}\xi\,H_{S^i}(f^i_\epsilon(t,0,\xi),\xi)\;\text{d}\xi\le 0,\quad \text{a.e. }t>0,
\end{equation}
for every kinetic solution $f_\epsilon$ to $\Psi^c$ and every $\mathcal S$ with (\ref{eq:CouplingNetworkConditioncijandS}).
We set $S^i(v)=1$ and $S^i(v)=v^2/2$ in (\ref{eq:EntropyFluxInequalityNetworkKinetic}) and get conservation of mass and energy at the junction. 
After setting $\omega_{\operatorname{min}}^1=\dots=\omega_{\operatorname{min}}^d=-\omega_{\operatorname{max}}^1=\dots=-\omega_{\operatorname{max}}^d$ we apply Theorem \ref{thm:ExistenceMacroscopicEquation} and obtain for the macroscopic solution $(\rho^i,u^i)$:
\begin{equation}\label{eq:EntropyFluxInequalityNetworkMacro}
\sum_{i=1}^d A^i \overline{G_{S^i}(\rho^i,u^i)}(t,0)\le 0,\quad \text{a.e. }t>0,
\end{equation}
for every $\mathcal S$ with (\ref{eq:CouplingNetworkConditioncijandS}). 
\subsection{Convolutional Coupling Conditions}
\label{ConvolutionalCouplingConditions}
We present coupling conditions defined by a convolution operator. For $a^{ij}\in L^1_\mu((0,\infty)_\xi,L^\infty(-\infty,0)_{\xi'})$, $i,j=1,\dots,d$, we define
\begin{equation}
\Psi^{a,i}[t,g](\xi)=\sum_{j=1}^d\int_{-\infty}^0 |\xi'|\, a^{ij}(\xi,\xi')\begin{pmatrix}g_0^j(\xi')\\-g_1^j(\xi')\end{pmatrix}\;\text{d}\xi'.
\end{equation}
Notice that the limit case $a^{ij}(\xi,\cdot)=\tfrac{c^{ij}}{\xi} \delta_\xi(\cdot)$ gives the coupling function in (\ref{eq:DefCouplingConditionLinear}). In contrast to this special case and (\ref{eq:CouplingNetworkConditioncijandS}), we are not able to prove similar entropy flux inequalities under possibly additional restrictions on $\mathcal S$. Nevertheless, (\ref{eq:PsiMassConservation}) and a scaling argument imply $b_0=0$ and
\begin{equation}
\sum_{i=1}^d A^i \int_0^\infty \xi\, a^{ij}(\xi,\xi')\;\text{d}\xi=A^j,\quad \text{for all }\xi'<0.
\end{equation}
\subsection{Maxwellian Boundary Conditions}
This subsection is devoted to restoring the results from \cite{BeBo2002Boundary} for initial boundary value problems
\begin{equation}
\rho(t,0)=\rho^b(t),\quad \rho(t,0)u(t,0)=\rho^b(t)u^b(t),\quad t>0.
\end{equation}
Since the seminal paper by Dubois and LeFloch \cite{DuLe1988}, it is well-known that this problem is overdetermined and we have to use the weaker boundary conditions
\begin{equation}\label{eq:EntropyFluxInequalityBCMacro}
\overline{G_S(\rho,u)}-G_S(\rho^b,u^b)-\eta'_S(\rho^b,u^b)\cdot (\overline{F(\rho,u)}-F(\rho^b,u^b))\le 0,\quad \text{a.e. }t>0.
\end{equation}
We choose $d=1$ and 
\begin{equation}
\Psi^b[t,g](\xi)=M(\rho^b(t),u^b(t),\xi),\quad \xi>0,
\end{equation}
with $(\rho^b,u^b)\in L^\infty((0,\infty)_t,\tilde D^i)$. Proposition \ref{prop:SubdifferentialInequality} with equality for $\xi>0$ gives
\begin{align}
\begin{split}\label{eq:EntropyFluxInequalityBCKinetic}
\int_{\mathbb R}\xi\,H_S(f_\epsilon(t,0,\xi),\xi)\;\text{d}\xi\le& \int_{\mathbb R}\xi\,H_S(M(\rho^b(t),u^b(t),\xi),\xi)\;\text{d}\xi\\
&+T_S(\rho^b,u^b)\int_{\mathbb R}\xi\,(f_\epsilon(t,0,\xi)-M(\rho^b(t),u^b(t),\xi)\;\text{d}\xi,
\end{split}
\end{align}
for every kinetic solution $f_\epsilon$ to $\Psi^b$. The existence follows from Theorem \ref{thm:ExistenceMacroscopicEquation} and (\ref{eq:EntropyFluxInequalityBCMacro}) follows from (\ref{eq:EntropyFluxInequalityBCKinetic}) and Corollary \ref{cor:MonotonicityGammaLimit}.
\subsection{Solid Wall Boundary Conditions}
Solid wall boundary conditions can be modeled by the special case of (\ref{eq:DefCouplingConditionLinear}) with $d=1$ and $c^{11}=1$. The coupling function is 
\begin{equation}\label{eq:DefinitionSolidWall1}
\Psi^w[t,g](\xi)=\begin{pmatrix}g_0(-\xi)\\-g_1(-\xi)\end{pmatrix},\quad \text{for }\xi>0.
\end{equation}
The macroscopic boundary traces satisfy
\begin{align}
\overline{G_S(\rho,u)}(t,0)\le 0,\quad \text{a.e. } t>0,
\label{eq:MacroscopicSolidWallBC}
\end{align}
for every convex $S\in C^1(\mathbb R)$ with $S(v)=S(-v)$ and $|S(v)|\le B(1+v^2)$ for all $v\in \mathbb R$. In particular, we have
\begin{equation}
\overline{\rho u}(t,0)=0,\quad \text{a.e. } t>0.
\end{equation}
Another way to introduce solid wall boundary conditions is
\begin{align}
\begin{split}
&\Psi^{w'}[t,g](\xi)=M(\rho^{w},0,\xi),\quad \xi>0, \text{ where }\rho^w\ge 0 \text{ with}\\
&\int_0^\infty |\xi|\,M_0(\rho^{w},0,\xi)\;\text{d}\xi=\int_0^\infty |\xi|\,g_0(-\xi)\;\text{d}\xi. 
\end{split}
\end{align}
First notice that one can easily check that this definition is well-defined and different to the coupling condition (\ref{eq:DefHatPsi}) with $\Psi=\Psi^w$. Since Proposition \ref{prop:SubdifferentialInequality} and the definition of $\rho^w$, we have
\begin{align*}
\int_{\mathbb R}\xi\,H_S(f_\epsilon(t,0,\xi),\xi)\;\text{d}\xi&\le G_S(\rho^w(t),0)+T_S(\rho^w(t),0)\int_{-\infty}^0 \xi\,\Big(f(\xi)-M(\rho^w(t),0,\xi)\Big)\;\text{d}\xi\\
&=0,
\end{align*}
for every convex $S\in C^1(\mathbb R)$ with $S(v)=S(-v)$ and $|S(v)|\le B(1+v^2)$ for all $v\in \mathbb R$.
Again, we get (\ref{eq:MacroscopicSolidWallBC}) after applying Theorem \ref{thm:ExistenceMacroscopicEquation} with $\omega_1=-\omega_2$.
\subsection{Nozzles with discontinuous cross-sections}
Our results can be used to study pipelines or nozzles with discontinuous cross-section. Usually these problems are solved by an approach called non-conservative products introduced by Dal Maso, LeFloch and Murat \cite{DLM1995}, but these tools require $BV$-regularity of the solutions. We can tackle this problem by setting $d=2$ in the results of Section \ref{sec:MainResults} after a variable transformation on the second pipeline. For non-conservative products one has some freedom in picking different Lipschitz-paths, which give different coupling conditions at the discontinuity. We have a similar phenomenon in our approach: In most of the applications we expect $b_0=b_H=0$ in (\ref{eq:PsiMassConservation} -- \ref{eq:PsiEntropyConservation}) and equality in the mass constrained (\ref{eq:PsiMassConservation}). Now, we can use the arguments in the Subsections \ref{sec:MaxwellianCouplingConditions} -- \ref{ConvolutionalCouplingConditions} to construct many different coupling conditions which satisfy these assumptions. 
\section{Extensions and Outlook}
\label{sec:ExtensionsandOutlook}
\subsection{Networks with arbitrary many junctions}
We want to show how to deal with networks consisting of $m\in\mathbb N$ junctions and $d\in\mathbb N$ pipelines since some modifications are necessary. Notice, that networks with arbitrary many junctions may contain circles. These circles can possibly lead to circulations with increasing speed such that the speed goes to infinity after finite time. We will show that this does not occur if we assume to have kinetic invariant domains.\\
First, we introduce some new notation. A pipeline is modeled by a compact, non-empty interval $[a^i_-,a^i_+],\,i=1,\dots,d,\,a^i_\pm\in\mathbb R$ (\textit{Remark:} The following analysis can be extended to closed intervals). Every pipeline is connected to exactly one junction at each end $a^i_-$ and $a^i_+$ and the functions $\theta_-,\theta_+\colon \{1,\dots,d\}\to \{1,\dots,m\}$ give the junctions at $a_-$ and $a_+$. The sets $T_-(k),T_+(k)\subset\{1,\dots,d\}$ are the sets of pipelines $i$ which are connected to the junction $k=1,\dots,m$ at $a^i_-$ and $a^i_+$ or equivalently $T_\pm(k)=\theta_\pm^{-1}(\{k\})$. Sometimes we use the index $\pm$ to treat the cases $+$ and $-$ together and we write $\sum_\pm$ for the sum of both cases.\\
We couple the kinetic solutions $f^i$ by
\begin{align}
\begin{split}
f^i(t,a^i_-,\xi)=\Psi_-^{\theta_-(i)}[t,f^j(t,a^j_\pm,\cdot); \,j=1,\dots,d](\xi),\quad \xi>0,\\
f^i(t,a^i_+,\xi)=\Psi_+^{\theta_+(i)}[t,f^j(t,a^j_\pm,\cdot); \,j=1,\dots,d](\xi),\quad \xi<0.
\end{split}
\end{align} 
The coupling functions $\Psi^{k}$ are defined by
\begin{align}\label{eq:defPsimorethanonejunction}
\begin{split}
\Psi^k\colon (0,\infty)_t\times L_\mu^1((-\infty,0)_\xi,D)^{d}\times L_\mu^1((0,\infty)_\xi,D)^{d}&\to L_\mu^1((0,\infty)_\xi,D)^{d}\times L_\mu^1((-\infty,0)_\xi,D)^{d};\\
[t,g_-,g_+]&\mapsto (\Psi^k_-,\Psi^k_+),\\
\text{where $\Psi^k$ depends only on $g_\pm^i$ with $\theta_\pm(i)=k$}& \text{ and $\Psi^{k,i}_{\pm}[t,g_-,g_+]=0$ if $\theta_\pm(i)\neq k$}.
\end{split}
\end{align}
They satisfy the continuity property:
\begin{equation}
\begin{split}
L_\mu^1((0,\infty)_{\mathrm{loc},t}\times &(-\infty,0)_\xi,D)^{d}\times L_\mu^1((0,\infty)_{\mathrm{loc},t}\times(0,\infty)_\xi,D)^{d}\to\\
L_\mu^1((0,\infty)_{\mathrm{loc},t}&\times(0,\infty)_\xi,D)^{d}\times L_\mu^1((0,\infty)_{\mathrm{loc},t}\times(-\infty,0)_\xi,D)^{d};\\
g\mapsto \Big((t,\xi)&\mapsto(\Psi_-,\Psi_+)[t,g_-(t,\cdot),g_+(t,\cdot)](\xi)\Big)\quad \text{ is continuous.}
\end{split}\label{eq:ContPsimorethanonejunction}
\end{equation}
The conditions (\ref{eq:PsiMassConservation}), (\ref{eq:PsiEntropyConservation}) and (\ref{eq:KineticInvariancePsi}) can be generalized in the following way. There exists $b_{\mathcal S_-,\mathcal S_+}^k\in L^1_{\mathrm{loc},t}(0,\infty)$ such that 
\begin{align}\label{eq:ConditionsMassEntropyJunctionGeneralNetwork}
\begin{split}
\sum_{i=1}^{d}A^i\int_0^\infty |\xi| H_{S^i_-}(\Psi_-^{k,i}[t,g_-,g_+](\xi)\;\text{d}\xi+\sum_{i=1}^{d}A^i\int_{-\infty}^0 |\xi| H_{S^i_+}(\Psi_+^{k,i}[t,g_-,g_+](\xi)\;\text{d}\xi\\
\le \sum_{i=1}^d A^i \int_{-\infty}^0 |\xi| H_{S^i_-}(g_-(\xi),\xi)\;\text{d}\xi+\sum_{i=1}^d A^i \int_0^\infty |\xi| H_{S^i_+}(g_+(\xi),\xi)\;\text{d}\xi+b_{\mathcal S_-,\mathcal S_+}^k(t),
\end{split}
\end{align}
for a.e. $t\in (0,\infty)$, for $\mathcal S_-=\mathcal S_+ \in \{\mathcal S_0=(1,\dots,1),\,\mathcal S_H=(v^2/2,\dots,v^2/2),\,\mathcal S_\omega=(S_\omega^1,\dots,S_\omega^d)\}$, where $S_\omega^i(v)=(v-\omega_{\max}^i)_+^2+(\omega_{\min}^i-v)_+^2$, $b_{\mathcal S_\omega,\mathcal S_\omega}=0$ and $(g_-,g_+)\in  L_\mu^1((-\infty,0)_\xi,D)^d\times L_\mu^1((0,\infty)_\xi,D)^d$.
\begin{theorem}
Let $f^{0,i}\in L^1((a^i_-,a^i_+)_x\times \mathbb R_\xi,\tilde D_\xi^i)$ with $\iint_{(a^i_-,a^i_+)\times\mathbb R}H(f^{0,i}(x,\xi),\xi)\;\text{d}x\text{d}\xi<\infty$. Let $\Psi$ be such that (\ref{eq:ContPsimorethanonejunction}) holds and (\ref{eq:ConditionsMassEntropyJunctionGeneralNetwork}) holds for $\mathcal S_-=\mathcal S_+\in \{\mathcal S_1,\mathcal S_H,\mathcal S_\omega\}$, $b_{\mathcal S_\omega,\mathcal S_\omega}=0$.
Then, there exist coupled BGK solutions $f^i_\epsilon$ to $\Psi$ for every $\epsilon>0$.
After passing if necessary to a subsequence $(\rho_\epsilon^i,\rho_\epsilon^i u_\epsilon^i)$ converge a.e. to an entropy solution $(\rho^i,\rho^i u^i)$ to the isentropic gas equations with initial data $(\rho^{0,i},\rho^{0,i} u^{0,i})=\int_\mathbb R f^{0,i}\text{d}\xi$. \\
Furthermore, after passing if necessary to a subsequence again, we have
\begin{equation}
\pm \overline{G_S(\rho^i,u^i)}(t,a_\pm^i)\ge \pm\psi_S^{i,\pm}(t):=\pm \underset{\epsilon\to 0}{\operatorname{w*-lim}}\int_\mathbb R \xi\,H_S(f_\epsilon(t,a_\pm^i,\xi),\xi)\;\text{d}\xi
\end{equation}
a.e. $t>0$, where $S\colon \mathbb R\to\mathbb R$ is convex, of class $C^1$ and $|S(v)|\le B(1+v^2)$ for a constant $B$. 
\end{theorem}
\begin{proof}
We use the same arguments as we used to prove Theorem \ref{thm:ExistenceMacroscopicEquation}, but we have to modify two parts.\\
\underline{Part 1:} The first part is the stability estimate after (\ref{eq:stabilityestimatedifferentformorethanonejunction}). Since 
\begin{equation*}
M[g_n^i]\to M[g^i] \quad \text{in }L^1((0,T)_t\times (a_-^i,a_+^i)_{\mathrm{(loc)},x}\times \mathbb R_\xi),
\end{equation*}
we can handle all integrals containing $|M[g_n^i]-M[g^i]|$ easily and just denote all of them by $\delta(M[g_n])$ for simplicity. By the characteristics formula, we have
\begin{align*}
&\iint_{(a_-^i,b_+^i)\times\mathbb R}|F^i(g_n)-F^i(g)|(t,x,\xi)\,\mathbbm 1_{\pm \xi<0}\;\text{d}x\text{d}\xi\\
&\le \iint_{(0,t)\times\mathbb R} |\xi|\,|\Psi_\pm^{\theta_\pm(i)}[t,F(g_n)(t,a^i_\pm,\cdot)]-\Psi_\pm^{\theta_\pm(i)}[t,F(g)(t,a^i_\pm,\cdot)]|(\xi)\,\mathbbm 1_{\{\pm \xi<0\}}\,\text{d}x\text{d}\xi+\delta(M[g_n]),
\end{align*}
Since (\ref{eq:ContPsimorethanonejunction}), it remains to prove
\begin{align*}
\sum_{i=1}^d \sum_{\pm} \iint_{(0,T)\times\mathbb R} |\xi|\,|F(g_n)(t,a^i_\pm,\xi)-F(g)(t,a^i_\pm,\xi)|\,\mathbbm 1_{\{\pm \xi>0\}}\,\text{d}x\text{d}\xi\to 0.
\end{align*}
Because we assumed to have kinetic invariant domains, we get by the characteristics formula
\begin{align}
&\sum_{i=1}^d \sum_{\pm} \iint_{(0,T)\times\mathbb R} |\xi|\,|F(g_n)(t,a^i_\pm,\xi)-F(g)(t,a^i_\pm,\xi)|\,\mathbbm 1_{\{\pm \xi>0\}}\,\text{d}x\text{d}\xi \label{eq:morethanonejunctioncontinuityhilf}\\
&\le \sum_{i=1}^d \sum_{\pm}\iint_{(0,T-\Delta)\times\mathbb R} |\xi|\,|\Psi_\pm^{\theta_\pm(i)}[t,F(g_n)(t,a^i_\pm,\cdot)]-\Psi_\pm^{\theta_\pm(i)}[t,F(g)(t,a^i_\pm,\cdot)]|(\xi)\,\mathbbm 1_{\{\pm \xi<0\}}\,\text{d}x\text{d}\xi+\delta(M[g_n]),\nonumber
\end{align}
with
\begin{equation}
\Delta=\inf_{i} \frac{a_+^i-a_-^i}{\max\{|\omega_{\min}^i|,|\omega_{\max}^i|\}}>0.
\end{equation}
We do $\lceil T/\Delta\rceil$ iterations of the estimate (\ref{eq:morethanonejunctioncontinuityhilf}) and use (\ref{eq:ContPsimorethanonejunction}) to prove the desired stability result.\\
\underline{Part 2:} Additionally, we have to modify the estimates (\ref{eq:EstimateHwithBC1} -- \ref{eq:EstimateHwithBC3}). By the characteristics formula and Jensen's inequality, we get
\begin{align*}
&\sum_{i=1}^d A^i \iint_{(a^i_-,a^i_+)\times\mathbb R}H(F^i(g)(t,x,\xi),\xi)\;\text{d}x\text{d}\xi\\
&\quad+\sum_{i=1}^d\sum_\pm A^i\iint_{(0,t)\times\mathbb R}|\xi| H(F^i(g)(s,a^i_\pm,\xi),\xi)\,e^{(s-t)/\epsilon}\mathbbm 1_{\{\pm\xi>0\}}\;\text{d}t\text{d}\xi\\
&\le \sum_{i=1}^d A^i \iint_{(a^i_-,a^i_+)\times\mathbb R}H(f^{0,i}(t,x,\xi),\xi)e^{-t/\epsilon} \;\text{d}x\text{d}\xi\\
&\quad+\sum_{i=1}^d\sum_\pm A^i\iint_{(0,t)\times\mathbb R}|\xi| H(F(g)(s,a^i_\pm,\xi),\xi)\,e^{(s-t)/\epsilon}\mathbbm 1_{\{\pm\xi<0\}}\;\text{d}t\text{d}\xi\\
&\quad+\sum_{i=1}^d \frac{A^i}{\epsilon}\iiint_{(0,t)\times (a_-^i,a_+^i)\times \mathbb R} H(M[g^i](s,x,\xi),\xi)e^{(s-t)/\epsilon}\;\text{d}s\text{d}x\text{d}\xi.
\end{align*}
The entropy bound (\ref{eq:ConditionsMassEntropyJunctionGeneralNetwork}) with $S^i_\pm=v^2/2$ and integration by parts imply
\begin{align*}
&\sum_{i=1}^d A^i \iint_{(a^i_-,a^i_+)\times\mathbb R}H(F^i(g)(t,x,\xi),\xi)\;\text{d}x\text{d}\xi\\
&\le  \sum_{i=1}^d A^i \iint_{(a^i_-,a^i_+)\times\mathbb R}H(f^{0,i}(t,x,\xi),\xi)e^{-t/\epsilon} \;\text{d}x\text{d}\xi+\int_0^t b_{\mathcal S_H,\mathcal S_H}(s)\;\text{d}s.
\end{align*}
But this is the generalized version of (\ref{con:C2}) and we get the result.
\end{proof}\pagebreak
\begin{remark}~
\begin{itemize}
\item We used the kinetic invariant domains in part 1 of the proof. Therefore, the generalized version of Theorem \ref{thm:ExistenceKineticModel} is weaker than the original one.
\item We give a generalization of Corollary \ref{cor:MonotonicityGammaLimit}: Let $p\in \mathbb N$, $i_l\in\{1,\dots,d\}$, $S_l\colon\mathbb R\to\mathbb R$ convex, of class $C^1$ with $|S_l(v)|\le B_l(1+v^2)$, $l=1,\dots,p$.
Let $\Gamma\colon (0,\infty)_t\times \mathbb R^p \times \mathbb R^p\to \mathbb R$ be such that $\Gamma[t,\cdot,\cdot]$ is uniformly bounded on compact sets. Furthermore, let $\Gamma[t,\cdot,G_+]$ (resp. $\Gamma[t,G_-,\cdot]$) be increasing (resp. decreasing) in every argument with $S_l\notin \operatorname{span}\{\mathbbm 1,v\}$. Then, 
\begin{equation}
\Gamma[t,\overline{G_{S_1}(\rho^{i_1},u^{i_1})}(t,a_\pm^{i_1}),\dots,\overline{G_{S_p}(\rho^{i_p},u^{i_p})}(t,a_\pm^{i_p})]\le \Gamma[t,\psi_{S_1}^{i_1,\pm}(t),\dots,\psi_{S_p}^{i_p,\pm}(t)]\le b_{\Gamma, \mathcal S}(t),\label{eq:CouplingConditionmorethanonejunctionGamma}
\end{equation}
a.e. $t>0$, where $b_{\Gamma, \mathcal S}\in L^\infty_t(0,\infty)$ depends only on $\Psi$
\item Notice that (\ref{eq:CouplingConditionmorethanonejunctionGamma}) can be decomposed to local inequalities at the juncions $k=1,\dots,m$.
\end{itemize}
\end{remark}
\subsection{Further Generalizations and Outlook}
\subsubsection*{Non-local in time coupling conditions}
We considered coupling conditions local in time, what means that $\Psi[t,f_\epsilon]$ depend only on $f_\epsilon(t,0,\xi)$. Our arguments can be adapted to the more general case that $\Psi[t,f_\epsilon]$ depend on $f_\epsilon(s,0,\xi),\,s\in[0,t]$. This allows to model the case that gas entering the junction at a certain time leaves the junction at a later time. It turns out that (\ref{eq:PsiMassConservation} -- \ref{eq:PsiEntropyConservation}) and (\ref{eq:KineticInvariancePsi}) are still sufficient to prove convergence. Notice that $b_{\mathcal S}(t)$ in (\ref{eq:DefboundbS}) is a bound for the entropy leaving at $t$, but enters the junction at an possibly earlier time. Therefore, $b_{\mathcal S}$ can be a very bad bound and it seems to be necessary to introduce more precise conditions. Otherwise we can not expect to get (in some sense) uniqueness for the macroscopic problem.
\subsubsection*{Omitting the $L^\infty$-bounds}
As shown in \cite{LeWe2007}, it is possible to omit the $L^\infty$-bounds on the initial data to get existence of finite mass and energy solutions to the isentropic gas equations on the full line (with $1<\gamma<5/3$). The solutions are constructed by the limit of solutions with bounded initial data. The key problem in adapting these techniques to networks is to approximate the coupling condition $\Psi$ by $\Psi^n$, where $\Psi^n$ admits a family of kinetic invariant domains. In some cases we get this naturally by setting $\Psi^n=\Psi$ (for example (\ref{eq:DefCouplingConditionLinear})). Furthermore, we need a generalization of Theorem \ref{thm:WeakTraceDivMeasure} for equi-integrable solutions.
\subsubsection*{Outlook}
Our results could be used to study and justify numerical methods which use the kinetic BGK model. Furthermore, these techniques could be adapted to other hyperbolic equations with kinetic models, but notice that the rich family of entropies is very important to pass to the macroscopic limit. 
We obtained entropy-flux inequalities at the junction for our macroscopic limit. It is an interesting question if these inequalities ensure uniqueness of the solutions or at least in some special cases. On the other hand one could study if different kinetic coupling conditions converge to the same macroscopic limit and one could try to characterize the obtained equivalence classes.

\section{Appendix}
\label{sec:Appendix}
We recall an existence result for weak traces of divergence measure fields \cite{An1983,ChFr1999}.
\begin{theorem}\label{thm:WeakTraceDivMeasure}
Let $V=(V_0,V_1)\in L^\infty((0,\infty)_t\times(0,\infty)_x)$ be a vector field such that $\operatorname{div}_{t,x}V\in\mathcal M((t_1,t_2)_t\times (0,R)_x)$ for any $0<t_1<t_2<\infty$ and $R>0$. Then there exists a unique solution $\overline V_1\in L^\infty_t(0,\infty)$ to
\begin{gather}
-\iint_{(0,\infty)^2}\varphi \operatorname{div}V-\iint_{(0,\infty)^2}V_0\,\partial_t \varphi\;\text{d}t\text{d}x-\iint_{(0,\infty)^2}V_1\,\partial_x\varphi\;\text{d}t\text{d}x\nonumber\\
-\int_0^\infty \overline V_1 \,\varphi(t,0)\;\text{d}t=0
\end{gather}
for any $\varphi\in \mathcal D((0,\infty)_t\times [0,\infty)_x)$. In fact $\overline V_1$ depends only on $V_1$ and satisfies $\norm{\overline V_1}_{L^\infty}\le\norm{V_1}_{L^\infty}$.
\end{theorem}
%
\bibliographystyle{plain}
\bibliography{99_Bibliography}
\end{document}